\documentclass[12pt,oneside]{amsart}
\usepackage{latexsym,amssymb,amsmath,graphicx,hyperref,tikz,color}
\usepackage{verbatim}
\numberwithin{equation}{section}
\newtheorem{theorem}{Theorem}[section]
\newtheorem{lemma}[theorem]{Lemma}

\newtheorem{corollary}[theorem]{Corollary}
\newtheorem{conjecture}[theorem]{Conjecture}

\theoremstyle{definition}
\newtheorem{definition}[theorem]{Definition}
\newtheorem{remark}[theorem]{Remark}
\newtheorem{example}[theorem]{Example}

\newcommand{\Ass}{{\rm Ass}}
\newcommand{\bipyr}{{\rm bipyr}}
\newcommand{\conv}{{\rm conv}}
\newcommand{\ffi}{\varphi}

\begin{document}

\title{The Waldschmidt constant for squarefree monomial ideals}
\thanks{Last updated: Revised Version (May 21, 2016)}

\author{Cristiano Bocci}
\address{Department of Information Engineering and Mathematics,
University of Siena\\
Via Roma, 56 Siena, Italy}
\email{cristiano.bocci@unisi.it}

\author{Susan Cooper}
\address{Department of Mathematics\\
North Dakota State University\\
NDSU Dept \#2750\\
PO Box 6050\\
Fargo, ND 58108-6050, USA}
\email{susan.marie.cooper@ndsu.edu}

\author{Elena Guardo}
\address{Dipartimento di Matematica e Informatica\\
Viale A. Doria, 6 \\
95100 - Catania, Italy} \email{guardo@dmi.unict.it}

\author{Brian Harbourne}
\address{Department of Mathematics\\
University of Nebraska\\
Lincoln, NE 68588-0130, USA}
\email{bharbourne1@unl.edu}

\author{Mike Janssen}
\address{Department of Mathematics, Statistics, and Computer Science\\
 Dordt College, Sioux Center, IA 51250, USA}
\email{mike.janssen@dordt.edu}

\author{Uwe Nagel}
\address{Department of Mathematics\\
University of Kentucky\\
715 Patterson Office Tower\\
Lexington, KY 40506-0027, USA}
\email{uwe.nagel@uky.edu}

\author{Alexandra Seceleanu}
\address{Department of Mathematics\\
University of Nebraska\\
Lincoln, NE 68588-0130, USA}
\email{aseceleanu@unl.edu}

\author{Adam Van Tuyl}
\address{Department of Mathematics and Statistics,
McMaster University, Hamilton, ON, L8S 4L8, Canada}
\email{vantuyl@math.mcmaster.ca}

\author{Thanh Vu}
\address{Department of Mathematics\\
University of Nebraska\\
Lincoln, NE 68588-0130, USA}
\email{tvu@unl.edu}


\begin{abstract}
Given a squarefree monomial ideal $I \subseteq R =k[x_1,\ldots,x_n]$,
we show that $\widehat\alpha(I)$, the Waldschmidt constant of $I$, can
be expressed as the optimal solution to a linear program constructed
from the primary decomposition of $I$. By applying results
from fractional graph theory, we can then express $\widehat\alpha(I)$
in terms of the fractional chromatic number of a hypergraph also
constructed from the primary decomposition of $I$.  Moreover, expressing
$\widehat\alpha(I)$ as the solution to a linear program
enables us to prove a Chudnovsky-like lower bound on $\widehat\alpha(I)$,
thus verifying a conjecture of Cooper-Embree-H\`a-Hoefel for monomial
ideals in the squarefree case.  
As an application, we compute the Waldschmidt constant and the resurgence
for some families of squarefree monomial ideals.  For example,
we determine both constants for unions of general linear subspaces of $\mathbb P^n$ with
few components compared to $n$, and we find the Waldschmidt constant
for the Stanley-Reisner ideal of a uniform matroid.
\end{abstract}

\maketitle


\section{Introduction}\label{sec:intro}

During the last decade, there has been a lot of interest in the
``ideal containment problem'': given a nontrivial homogeneous
ideal $I$ of a polynomial ring $R = k[x_1,\ldots,x_n]$ over a field $k$,
the problem is to determine all positive integer pairs $(m,r)$ such that
$I^{(m)} \subseteq I^r$. Here $I^{(m)}$ denotes the $m$-th symbolic power
of the ideal, while $I^r$ is the ordinary $r$-th power of $I$
(formal definitions are postponed until the next section).
This problem was motivated by the fundamental results of \cite{refELS, refHoHu}
showing that containment holds whenever $m\geq r(n-1)$.
In order to capture more precise information about these containments, Bocci
and Harbourne \cite{BH} introduced the {\em resurgence} of $I$, denoted
$\rho(I)$ and defined as $\rho(I) = \sup\{m/r ~|~ I^{(m)} \not\subseteq I^r \}$.

In general, computing $\rho(I)$ is quite difficult.  Starting
with \cite{BH}, there has been an ongoing
research programme to bound $\rho(I)$ in terms of other
invariants of $I$ that may be easier to compute.
One such bound is in terms of the
Waldschmidt constant of $I$.  Given any nonzero homogeneous ideal $I$
of $R$, we let $\alpha(I) = \min\{ d ~|~ I_d  \neq 0 \}$; i.e.,
$\alpha(I)$ is the smallest degree of a nonzero element in $I$.  The
{\it Waldschmidt constant} of $I$ is then defined to be
\[\widehat\alpha(I) = \lim_{m \rightarrow \infty} \frac{\alpha(I^{(m)})}{m}. \]

This limit exists and was first defined
by Waldschmidt \cite{W} for ideals of finite point sets in the context of complex analysis.  In
the language of projective varieties, Waldschmidt was interested
in determining the minimal degree of a hypersurface that passed through
a collection of points with prescribed multiplicities, that is,
he was interested in determining $\alpha(I^{(m)})$ when $I$ defined a set of
points.  Over the years, $\widehat\alpha(I)$ has appeared in many
guises in different areas of mathematics, e.g., in
number theory \cite{C,W,refW2}, complex analysis \cite{refSk},
algebraic geometry \cite{BH, BH2, refEV,  refSch} and commutative algebra
\cite{HH}.

Bocci and Harbourne's result $\alpha(I)/\widehat\alpha(I) \leq \rho(I)$
(see \cite[Theorem 1.2]{BH}) has renewed interest in computing
$\widehat\alpha(I)$.  For example, Dumnicki \cite{D} finds
lower bounds for $\widehat\alpha(I)$ when $I$ is an ideal of generic
points in $\mathbb{P}^3$,  Dumnicki, et al.\ \cite{DHNSST}
compute  $\widehat\alpha(I)$ when $I$ defines a set of points coming from
a hyperplane arrangement,
Fatabbi, et al.\ \cite{FHL} computed $\widehat\alpha(I)$
when $I$ defines a special union of linear varieties called
inclics,
 M. Baczy\'nska, et al.\ \cite{Betal}
examine $\widehat\alpha(I)$ when $I$ is a bihomogeneous ideal
defining a  finite sets of points
in $\mathbb{P}^1 \times \mathbb{P}^1$ in \cite{Betal}. Guardo,
et al.\ \cite{GHVT2} also  computed
$\widehat\alpha(I)$ when $I$ is the ideal
of general sets of points in $\mathbb{P}^1 \times \mathbb{P}^1$.
In addition, upper bounds on $\widehat\alpha(I)$ were studied in \cite{DHST,GHVT},
along with connections to Nagata's conjecture.
Even though computing $\widehat\alpha(I)$ may be easier than computing
$\rho(I)$, in general, computing the Waldschmidt constant remains a difficult
problem.

In this paper we focus on the computation of $\widehat\alpha(I)$ when
$I$ is a squarefree monomial ideal.  After reviewing the necessary
background in Section 2, in Section 3 we turn to
our main insight: that $\widehat\alpha(I)$
can be realized as the value of the optimal solution of a linear program
(see Theorem \ref{mainresult}).  To set up the required linear program,
we only need to know the minimal primary decomposition of the
squarefree monomial ideal $I$.   The Waldschmidt constant of monomial
ideals (not just squarefree) was first studied in \cite{CEHH}
(although some special cases can be found in \cite{BF,F}) which
formulates the computation of $\widehat\alpha(I)$ as a minimal value problem on a polyhedron
constructed from the generators of $I$.  Our contribution gives a more
effective approach using the well-known simplex method for computing the
Waldschmidt constant (see Remark \ref{alpharemark} for connections to \cite{CEHH}).

The ability to express $\widehat\alpha(I)$
as a solution to a linear program has a number of advantages.
First, in Section 4 we  relate $\widehat\alpha(I)$ to a combinatorial
invariant.  Specifically, we can view a squarefree
monomial ideal $I$ as the edge ideal of a hypergraph $H = (V,E)$
where $V = \{x_1,\ldots,x_n\}$ are the vertices and
$\{x_{j_1},\ldots,x_{j_t}\}$ is an edge (i.e., $\{x_{j_1},\ldots,x_{j_t}\}\in E$)
if and only if $x_{j_1}\cdots x_{j_t}$ is
a minimal monomial generator of $I$.
We then have the following result.

\begin{theorem}[{Theorem \ref{mainresult2}}]\label{them1}
Suppose that $H = (V,E)$ is a hypergraph with a non-trivial edge,
and let $I = I(H)$.  Then
$$\widehat\alpha(I)=\frac{\chi^*(H)}{\chi^*(H)-1}.$$
where $\chi^*(H)$ is the fractional chromatic number of the hypergraph
$H$.
\end{theorem}

Because the fractional chromatic number of a (hyper)graph is a well-studied
object (e.g., see the book \cite{SU}), Theorem \ref{them1} enables
us to utilize a number of known graph theoretic results to compute
some new values of $\widehat\alpha(I)$. For example,
in Section 6 we compute
$\widehat\alpha(I)$ when $I$ is an edge ideal for various well-known
families of graphs (e.g., bipartite, perfect, cycles).  We also show
how to simplify the proof of the main result of \cite{BF,F}. Moreover, we establish that the Waldschmidt constant of the edge ideal of a graph admits a lower and an upper bound in terms of the chromatic number and the clique number of the graph, respectively.

Second, the reformulation of $\widehat\alpha(I)$ as a linear program
gives us a new proof technique that allows us to prove a
Chudnovsky-like lower bound on $\widehat\alpha(I)$ in Section 5.
Chudnovsky \cite{C}
originally proposed a conjecture on $\widehat\alpha(I)$ in terms of $\alpha(I)$
and $n$ when $I$ defined a set of points in $\mathbb{P}^n$.  Cooper, et al. \cite{CEHH}
proposed a Chudnovsky-like lower bound for all monomial ideals.  We
verify this conjecture in the squarefree case:

\begin{theorem}[{Theorem \ref{mainresult3}}]
Let $I$ be a squarefree monomial ideal with big-height$(I) = e$.
Then
\[\widehat\alpha(I) \geq \frac{\alpha(I)+e-1}{e}.\]
\end{theorem}

\noindent
We give an example to show that this lower bound is sometimes sharp.

In Section 7, we 
illustrate how our new technique leads to new containment results, thus returning to the initial
motivation for studying Waldschmidt constants.  In particular,
in this section we study unions of a small number of general linear varieties,
the Stanley-Reisner ideal of a uniform matroid, and a ``monomial star'',
a squarefree monomial ideal of mixed height.

Although we have only focused on squarefree monomial ideals in this paper, 
our work has implications
for the ideal containment problem for a much larger class of ideals.  In particular,
recent work of  Geramita, et al. \cite{GHMN} has shown, among other things, that
if $\tilde{I}$ is a specialization of a monomial ideal $I$, 
i.e. $\tilde{I}$ is obtained by replacing each variable by a homogeneous
polynomial with the property that these polynomials form a regular sequence, then $\widehat\alpha(\tilde{I})$ and/or $\rho(\tilde{I})$ can be related
to $\widehat\alpha(I)$ and/or $\rho(I)$ of the monomial ideal (see, for example,
\cite[Corollary 4.3]{GHMN}).

\noindent {\bf Acknowledgements.} This project was started at the
Mathematisches Forschungsinstitut Oberwolfach (MFO) as part of the
mini-workshop ``Ideals of Linear Subspaces, Their Symbolic Powers
and Waring Problems'' organized by C. Bocci, E. Carlini, E.
Guardo, and B. Harbourne.     All the authors thank the MFO for
providing a stimulating environment. Bocci acknowledges the
financial support provided by GNSAGA of Indam. Guardo acknowledges
the financial support provided by PRIN 2011.
Harbourne was partially supported by NSA grant number  H98230-13-1-0213.
Janssen was partially
supported by Dordt College. Janssen and Seceleanu received support
from MFO's NSF grant DMS-1049268, ``NSF Junior Oberwolfach
Fellows''. Nagel was partially supported by the Simons Foundation
under grant No. 317096. Van Tuyl acknowledges the financial
support provided by NSERC.


\section{Background Definitions and Results}\label{sec:background}

In this section we review the relevant background.  Unless
otherwise indicated, $R = k[x_1,\ldots,x_n]$ with $k$ an
algebraically closed field of any characteristic. 
 We continue
to use the notation and definitions of the introduction.

\subsection{Squarefree monomial ideals and (hyper)graphs}
An ideal $I \subseteq R$ is a {\it monomial ideal} if $I$ is generated
by monomials.  We say that $I$ is a {\it squarefree
monomial ideal} if it is generated by squarefree monomials,
i.e., every generator has the form $x_1^{a_1}\cdots x_n^{a_n}$
with $a_i \in \{0,1\}$.
When $I$ is a squarefree monomial ideal, the minimal primary
decomposition of $I$ has the form
\[I = P_1 \cap \cdots \cap P_s ~~\mbox{with $P_i = 
\langle x_{j_1},\ldots,x_{j_{s_j}} \rangle$ for
$j=1,\ldots,s$}.\]

A {\it hypergraph} is an ordered pair $H = (V,E)$ where
$V = \{x_1,\ldots,x_n\}$ is the set of {\it vertices}, and
$E$ consists of subsets of $V$ such that if $e_i \subseteq e_j$,
then $e_i = e_j$.  The elements of $E$ are called {\it edges}.
When the hypergraph $H$ is such that $|e_i|=2$ for all $i$, it is also
called a {\it graph}.

Given any hypergraph $H = (V,E)$, we can associate to
$H$ a squarefree monomial ideal $I(H)$ called the {\it edge ideal}
of $H$.  Precisely,
\[I(H) = \langle x_{i_1}x_{i_2} \cdots x_{i_t} ~|~
\{x_{i_1},x_{i_2},\ldots,x_{i_t}\} \in E \rangle.\] This
construction can be reversed, so we have a one-to-one
correspondence between hypergraphs $H$ on $n$ vertices and
squarefree monomial ideals in $n$ variables.

\begin{remark} 
In the above one-to-one correspondence, we need to
be cognizant of the fact that a hypergraph with no edges is different than
a hypergraph whose edges are the isolated vertices.  In
the first case, $H = (V,\emptyset)$ is associated to the zero-ideal
$I(H) = (0)$, while in the second case, $H = (V,\{\{x_1\},\ldots,\{x_n\}\})$
is associated to $I(H) = \langle x_1,\ldots,x_n \rangle$.
In the first case, $\widehat{\alpha}((0))$ is not defined, while
in the second case, $\widehat{\alpha}(I(H)) = 1$ since $I(H)$
is generated by a regular sequence.  
Thus, it is harmless to eliminate these cases by considering 
only hypergraphs that have at least one non-trivial edge.” 
\end{remark}

The associated primes of $I(H)$ are related to the
maximal independent sets and vertex covers of the hypergraph
$H$. We say that $A \subseteq V$ is
an {\it independent set} of $H$ if $e \not\subseteq A$ whenever $e \in E$.
It is {\it maximal} if it is maximal with respect to inclusion.
A subset $U \subseteq V$ is a {\it vertex cover} of a hypergraph
if $e \cap U \neq \varnothing$ whenever $e \in E$.
A vertex cover is {\it minimal} if it is so with respect to
containment.

\begin{lemma}\label{vertexcovers}
Suppose that $H = (V,E)$ is a hypergraph with a non-trivial edge,
and let $I = I(H)$.
Suppose that $I = P_1 \cap
\cdots \cap P_s$ is the minimal primary decomposition of $I$, and
set $W_i = \{ x_j ~|~ x_j \not\in P_i\}$ for $i=1,\ldots,s$.  Then
$W_1,\ldots,W_s$ are the maximal independent sets of $H$.
\end{lemma}

\begin{proof}
Any $W$ is a maximal independent
set if and only if $V \setminus W$ is a minimal vertex cover.
We now use the fact that the
associated primes of the edge ideal $I(H)$ correspond to
the minimal vertex covers of $H$ (e.g., see
the proof \cite[Corollary 3.35]{VT}
for edge ideals of graphs, which can be
easily adapted to hypergraphs).
\end{proof}

\subsection{Symbolic Powers}
We now review the definition of symbolic powers. Recall that any homogeneous ideal $I \subseteq R$
has minimal primary decomposition $I = Q_1 \cap \cdots \cap Q_s$
where $\sqrt{Q_i} = P_i$ is a prime ideal.  The set of {\it
associated primes} of $I$, denoted ${\rm Ass}(I)$, is the
set ${\rm Ass}(I) = \{\sqrt{Q_i} ~|~ i=1,\ldots,s\}.$
The {\it minimal primes} of $I$, denoted ${\rm Min}(I)$, is
the set of minimal elements of ${\rm Ass}(I)$, ordered by
inclusion.

\begin{definition}
Let $0 \neq I \subseteq R$ be a homogeneous ideal.
The {\it $m$-th symbolic power} of $I$, denoted $I^{(m)}$, is
the ideal
\[
I^{(m)}= \bigcap_{P\in {\rm Ass}(I)}(I^mR_P \cap R),
\]
where $R_P$ denotes the localization of $R$ at the prime ideal
$P$.
\end{definition}

\begin{remark} In the literature, there is some ambiguity
concerning the notion of symbolic powers. The intersection in
the definition of the symbolic power is sometimes taken over all
associated primes and sometimes just over the minimal primes of $I$.
In general, these two possible definitions yield different results. However, they agree in the case of radical ideals, thus, in particular, also in the case of squarefree monomial ideals.
\end{remark}

We will be concerned with the analysis of generators of minimal
degree in the  symbolic powers $I^{(m)}$ of $I$.
While the general definition
of the $m$-th symbolic power of $I$ is based on localization,
for squarefree monomial ideals the following result will
prove useful.

\begin{theorem}\label{symbolicmonomial} 
Suppose that $I \subseteq R$ is a squarefree monomial ideal with
minimal primary decomposition $I = P_1 \cap \cdots \cap P_s$.
Then for all $m \geq 1$,
\[I^{(m)} = P_1^m \cap \cdots \cap P_s^m.\]
\end{theorem}

\begin{proof}This result is a special case of
\cite[Theorem 3.7]{CEHH}.
\end{proof}

The next result enables us to determine if a particular monomial
belongs to $I^{(m)}$.

\begin{lemma}\label{lem:monomialsinsymbolic}
Let $I \subseteq R$ be a squarefree monomial ideal with minimal
primary decomposition $ I = P_1 \cap P_2 \cap \cdots \cap P_s$
with  $P_j =\langle x_{j_1},\ldots,x_{j_{s_j}} \rangle$ for
$j=1,\ldots,s$.
Then $x_1^{a_1}\cdots x_n^{a_n} \in I^{(m)}$ if and only if
$a_{j_1} + \cdots + a_{j_{s_j}} \geq m$ for $j =1,\ldots,s$.
\end{lemma}

\begin{proof}
By Theorem \ref{symbolicmonomial}, $I^{(m)} = P_1^m \cap \cdots
\cap P_s^m$.  So $x_1^{a_1}\cdots x_n^{a_n} \in I^{(m)}$ if and
only if $x_1^{a_1}\cdots x_n^{a_n}$ is in $P_j^m$ for all $j=1,
\dots, s$. This happens if and only if there exists at least one
generator $f_j\in P_j^m$ such that $f_j$ divides $x_1^{a_1}\cdots
x_n^{a_n}$ (for $j=1, \dots, s$), which is equivalent to requiring
$a_{j_1} + \cdots + a_{j_{s_j}} \geq m$ for $j =1,\ldots,s$.
\end{proof}

\subsection{Waldschmidt constants}
We complete this section by reviewing some useful
results on $\widehat\alpha(I)$, the Waldschmidt constant of
a homogeneous ideal.

\begin{lemma}[Subadditivity]\label{subadditivity}
Let $I$  be a radical
homogeneous  ideal in $R = k[x_1,\ldots,x_n]$. Then
\begin{enumerate}
\item[$(i)$] $\alpha(I^{(c+d)}) \leq \alpha(I^{(c)}) + \alpha(I^{(d)})$
for all positive $c,d \in \mathbb{N}$.
\item[$(ii)$] $\widehat\alpha(I) = \lim_{m \rightarrow \infty} \frac{\alpha(I^{(m)})}{m}$ is the
infimum of $\alpha(I^{(m)})/m$ for $m \in \mathbb{N}$.
\end{enumerate}
\end{lemma}

\begin{proof}
The subadditivity of $\alpha(-)$ is a consequence of the fact that
symbolic powers of any radical homogeneous ideal form a graded system, meaning
that $I^{( c )}I^{( d )}\subseteq I^{( c+d) }$ for all $c,d\geq 0$
(see e.g \cite[Example 2.4.16 (iv)]{L}). The statement in part $(ii)$
then follows from $(i)$ by means of the general principle
of subadditivity in
\cite[Lemma A.4.1]{SU}. See \cite[Remark III.7]{HR} or
\cite[Lemma 2.3.1]{BH} for a version of the result in $(ii)$ and its proof.
Alternatively, use Fekete's Lemma \cite{Fe2} as in \cite{Betal}.
 \end{proof}


\section{The Waldschmidt constant and a linear program}\label{sec:results1}

When $I$ is a squarefree monomial ideal, we show that
$\widehat\alpha(I)$ can be expressed as the value to
a certain linear program arising from the
structure of the associated primes of $I$.
For the convenience of the reader, we review the relevant
definitions concerning linear programming (we have used \cite{SU}
for our reference).

A {\it linear program} (henceforth LP)
is a problem that can be expressed as:

\begin{tabular}{rl}
minimize & ${\bf b}^T{\bf y} $\\
subject to & $A{\bf y} \geq {\bf c}$ and  ${\bf y} \geq {\bf 0}$
\end{tabular}
\hspace{2cm}$(\star)$

\noindent where ${\bf b}$ is an $s$-vector, ${\bf c}$ is an
$r$-vector, ${\bf 0}$ is the zero $r$-vector, and $A$ is an $r
\times s$ real coefficient matrix.  Here, ${\bf d} \geq {\bf e}$
denotes the partial order where the $i$-th coordinate entry of
${\bf d}$ is larger than the $i$-th coordinate entry of ${\bf e}$
for all $i$.  Note that we wish to solve for the $s$-vector ${\bf y}$.
The equation ${\bf b}^T{\bf y}$ is the {\it constraint equation}.
Any ${\bf y}$ that satisfies $A{\bf y} \geq {\bf c}$ and  ${\bf y}
\geq {\bf 0}$ is called a {\it feasible solution}. If ${\bf y}^*$
is a feasible solution that optimizes the constraint equation,
then ${\bf b}^T{\bf y}^*$ is the {\it value} of  LP. Associated to
the LP $(\star)$ is its {\it dual linear program}:

\begin{tabular}{rl}
maximize & ${\bf c}^T{\bf x} $\\
subject to & $A^T{\bf x} \leq {\bf b}$ and  ${\bf x} \geq {\bf 0}$
\end{tabular}
\hspace{2cm}$(\star\star)$

\noindent A fundamental result in linear programming is that both
a linear program and its dual have the exact same value, i.e.,
${\bf c}^T{\bf x}^* = {\bf b}^T{\bf y}^*$ (see \cite[Theorem
A.3.1]{SU}). In particular, we shall find the following fact
useful.

\begin{lemma}\label{LPbounds} 
Consider the LP

\begin{tabular}{rl}
\emph{minimize} & ${\bf b}^T{\bf y} $\\
\emph{subject to} & $A{\bf y} \geq {\bf c}$ \emph{and}  ${\bf y} \geq {\bf 0}$
\end{tabular}

\noindent and suppose that ${\bf y}^*$ is the feasible solution
that gives the value of this LP.  If ${\bf x}$ is any feasible
solution of the associated dual LP, then ${\bf c}^T{\bf x} \leq
{\bf b}^T{\bf y}^*.$
\end{lemma}

\begin{proof}
For any feasible solution ${\bf x}$,
we have
\[{\bf c}^T{\bf x} = {\bf x}^T{\bf c} \leq {\bf x}^TA{\bf y}^*
= (A^T{\bf x})^T{\bf y}^* \leq {\bf b}^T{\bf y}^*.\]
\end{proof}

We now have the machinery to state and prove the first main result
of this paper.

\begin{theorem}
      \label{mainresult}
Let $I \subseteq R$ be a squarefree monomial ideal with minimal
primary decomposition $ I = P_1 \cap P_2 \cap \cdots \cap P_s$
with  $P_j =\langle x_{j_1},\ldots,x_{j_{s_j}} \rangle$ for
$j=1,\ldots,s$.
Let $A$ be the $s \times n$ matrix where
\[A_{i,j} = \begin{cases}
1 & \mbox{if $x_j \in P_i$} \\
0 & \mbox{if $x_j \not\in P_i$.}
\end{cases}
\]
Consider the following LP:

\begin{tabular}{rl}
\emph{minimize} & ${\bf 1}^T{\bf y} $\\
\emph{subject to} & $A{\bf y} \geq {\bf 1}$ \emph{and}  ${\bf y} \geq {\bf 0}$
\end{tabular}

\noindent
and suppose that ${\bf y}^*$ is a feasible solution that realizes
the optimal value.  Then
\[\widehat\alpha(I) = {\bf 1}^T{\bf y}^*.\]
That is, $\widehat\alpha(I)$ is the value of the LP.
\end{theorem}

\begin{proof}
Suppose that $({\bf y}^*)^T = \begin{bmatrix} y^*_1 & y^*_2 & \cdots & y^*_n \end{bmatrix}$ is the feasible solution that realizes the optimal
solution to the LP.  Because the entries of ${\bf y}^*$
are rational numbers, we can write
$({\bf y}^*)^T =
\begin{bmatrix}
\frac{a_1}{b_1} & \frac{a_2}{b_2} & \cdots & \frac{a_n}{b_n}
\end{bmatrix}$  with integers $a_i, b_i$ for $i=1,\ldots,n$.

Set $b = {\rm lcm}(b_1,\ldots,b_n)$.  Then $A(b{\bf y}) \geq {\bf b}$
where
${\bf b}$ is an $s$-vector of $b$'s.  So,
$(b{\bf y})$ is a feasible integer solution
to the system $A{\bf z} \geq {\bf b}$.
In other words, for each $j = 1,\ldots,s$,
\[b\left(\frac{a_{j_1}}{b_{j_1}}+ \cdots + \frac{a_{j_{s_j}}}{b_{j_{s_j}}}\right)
= \frac{b a_{j_1}}{b_{j_1}}+ \cdots + \frac{b a_{j_{s_j}}}{b_{j_{s_j}}}
\geq  b.\]
It then follows by Lemma \ref{lem:monomialsinsymbolic} that
\[x_1^{\frac{b a_1}{b_1}}x_2^{\frac{b a_2}{b_2}}\cdots x_n^{\frac{b a_n}{b_n}}
\in I^{(b)}.\]
Thus,
\[\alpha(I^{(b)}) \leq  \frac{b a_1}{b_1} +
\frac{b a_2}{b_2} + \cdots + \frac{b a_n}{b_n},\]
or equivalently (by Lemma \ref{subadditivity}),
\[\widehat\alpha(I) \leq
\frac{\alpha(I^{(b)})}{b} \leq  \frac{a_1}{b_1} + \frac{a_2}{b_2} +
\cdots + \frac{a_n}{b_n} = {\bf 1}^T{\bf y}^*.\]

To show the reverse inequality, suppose for a contradiction that
$\widehat\alpha(I) < {\bf 1}^T{\bf y^*}$.
By Lemma \ref{subadditivity}
we have
$\widehat\alpha(I) = \inf\left\{ {\alpha(I^{(m)})}/{m}\right\}_{m \in \mathbb{N}}.$
In particular, there must exist some $m$ such that
\[\frac{\alpha(I^{(m)})}{m} < \frac{a_1}{b_1} + \frac{a_2}{b_2} +
\cdots + \frac{a_n}{b_n} = {\bf 1}^T{\bf y}^*.\]
Let  $x_1^{e_1}x_2^{e_2}\cdots x_n^{e_n} \in I^{(m)}$ be a monomial
with $e_1+\cdots + e_n = \alpha(I^{(m)})$.  Then, by
Lemma \ref{lem:monomialsinsymbolic}, we have
\[e_{j_1} + \cdots + e_{j_{s_j}} \geq m ~~\mbox{for all $j=1,\ldots,s$}.\]
In particular, if we divide all the $s$ equations by $m$, we have
\[\frac{e_{j_1}}{m} + \cdots + \frac{e_{j_{s_j}}}{m} \geq 1
~~\mbox{for all $j=1,\ldots,s$}.
\]
But then
\[{\bf w}^T = \begin{bmatrix} \frac{e_{1}}{m} & \cdots & \frac{e_{s}}{m}
\end{bmatrix}^T\]
satisfies $A{\bf w} \geq {\bf 1}$ and ${\bf w} \geq {\bf 0}$.
In other words, ${\bf w}$ is a feasible solution to the LP, and furthermore,
$\frac{\alpha(I^{(m)})}{m} = {\bf 1}^T{\bf w}
<  \frac{a_1}{b_1} + \frac{a_2}{b_2} +
\cdots + \frac{a_n}{b_n} = {\bf 1}^T{\bf y}^*$, contradicting
the fact that ${\bf 1}^T{\bf y}^*$ is the value of the LP.
\end{proof}

\begin{remark}\label{alpharemark}
The set of feasible solutions of the LP in
Theorem \ref{mainresult} is the symbolic polyhedron for the
monomial ideal $I$ as defined in \cite[Definition 5.3]{CEHH}:
$$\mathcal{Q}=\bigcap\limits_{P\in \max\Ass(I)} {\rm conv \ }
\mathcal{L}(Q_{\subseteq P}).$$
Here, $Q_{\subseteq P}$ is the intersection of all primary
ideals $Q_i$ in the primary decomposition of $I$ with
$\sqrt{Q_i}\subseteq P$, $\mathcal{L}(Q_{\subseteq P})$ is the set
of lattice points $a \in \mathbb{N}^n$  with
$x^a = x_1^{a_1}\cdots x_n^{a_n} \in Q_{\subseteq P}$,
and ${\rm conv}(-)$ denotes the convex hull.
When $I$ is a squarefree monomial ideal $I$, then
$\Ass(I)=\max \Ass(I)$
and $Q_{\subseteq P}=P$ for any $P\in \Ass(I)$.  So we have
 $\mathcal{L}(Q_{\subseteq P})= \mathcal{L}( P )=
\{{\bf x} \mid {\bf x} \geq {\bf 0}, a_i \cdot {\bf x}\geq 1\}$,
where $a_i$ is the $i$-th row of the matrix $A$ in Theorem \ref{mainresult}.
Clearly then
 $\bigcap_i\mathcal{L}(Q_{\subseteq P})= \{{\bf x} \mid {\bf x} \geq {\bf 0},
 A  {\bf x}\geq {\bf 1}\}$.

 Furthermore, the optimal value of our LP is the same as
$\alpha( \mathcal{Q})$ as defined in \cite{CEHH} before Corollary 6.2,
thus Theorem \ref{mainresult} is a (useful!) restatement
(with easier proof) of \cite[Corollary 6.3]{CEHH}.
 \end{remark}

\begin{remark}
Because the set of optimal solutions to an integer LP consists of points
with rational coordinates, Theorem \ref{mainresult} allows us to conclude
that the Waldschmidt constant of any squarefree monomial ideal is rational.
The same is true for arbitrary monomial ideals by making use of the symbolic polyhedron described above.
\end{remark}


\section{The Waldschmidt constant in terms of a fractional chromatic number}

As shown in the last section, the Waldschmidt constant
$\widehat\alpha(I)$ of a squarefree monomial ideal $I
\subseteq R = k[x_1,\ldots,x_n]$ is the optimal
value of a linear program.  On the other hand, a squarefree
monomial ideal can also be viewed as the edge ideal
of a hypergraph $H = (V,E)$ where $V = \{x_1,\ldots,x_n\}$
and $\{x_{i_1},\ldots,x_{i_t}\} \in E$ is an edge
if and only if $x_{i_1}\cdots x_{i_t}$ is a minimal generator
of $I$.  We now show that $\widehat\alpha(I)$ can be expressed in
terms of a combinatorial invariant of $H$, specifically,
the fractional chromatic number of $H$.  Recall that we 
are assuming that all our hypergraphs $H = (V,E)$ have a non-trivial edge.

We begin by defining the fractional chromatic number of
a hypergraph $H = (V,E)$.
Set
\[\mathcal{W} = \{ W \subseteq V ~|~ \mbox{$W$ is an independent set of $H$}\}.\]

\begin{definition}\label{fracchromdefn}
Let $H = (V,E)$ be a hypergraph.  Suppose that
$\mathcal{W} = \{W_1,\ldots,W_t\}$ is the set of all independent
sets of $H$.
Let $B$ be the $n \times t$ matrix given by
\[B_{i,j} = \begin{cases}
1 & \mbox{if $x_i \in W_j$} \\
0 & \mbox{if $x_i\not\in W_j$.}
\end{cases}
\]
The optimal
value of the following LP, denoted $\chi^*(H)$,

\begin{tabular}{rl}
minimize & $y_{W_1} + y_{W_2} + \cdots + y_{W_t} = {\bf 1}^T{\bf y}$\\
subject to & $B{\bf y} \geq {\bf 1}$ \emph{and}  ${\bf y} \geq {\bf 0}$
\end{tabular}

\noindent
is the {\it fractional chromatic number of} the
hypergraph $H$.
\end{definition}

\begin{remark} If $H = (V,E)$ is a hypergraph with a non-trivial
edge, then $\chi^*(H)  \neq 1$.
\end{remark}

\begin{remark}
A {\it colouring} of a hypergraph $H = (V,E)$ is an assignment of
a colour to every $x \in V$ so that no edge is mono-coloured.
The minimum number of colours needed to give $H$ a valid colouring
is the {\it chromatic number} of $H$, and is denoted
$\chi(H)$. The value of $\chi(H)$ can also be interpreted as the
value of the optimal integer solution to the LP in the previous
definition.  In other words, the fractional chromatic number
is the relaxation of the requirement that the previous LP have
integer solutions.
\end{remark}

We next give a lemma which may be of independent interest due to its implications on computing fractional chromatic numbers. Our lemma shows that the dual of the LP that defines the fractional chromatic number can be reformulated in terms of a smaller number of constraints.

\begin{lemma}\label{maxindeplemma}
Let $H = (V,E)$ be a hypergraph.  Suppose that
$\mathcal{W'} = \{W_1,\ldots,W_s\}$ is the set of all maximal independent
sets of $H$.
Let $B'$ be the $n \times s$ matrix given by
\[B'_{i,j} = \begin{cases}
1 & \mbox{if $x_i \in W_j$} \\
0 & \mbox{if $x_i\not\in W_j$.}
\end{cases}
\] and let $B$ be the matrix defined in \ref{fracchromdefn}. Then the following two linear programs have the same feasible solution sets and the same optimal values:

\begin{tabular}{rl}
maximize & $w_1 + \cdots + w_n = {\bf 1}^T{\bf w}$ \\
subject to & $B^T{\bf w} \leq {\bf 1}$ \emph{and}  ${\bf w} \geq {\bf 0}$
\end{tabular}
 \qquad
\begin{tabular}{rl}
maximize & $w_1 + \cdots + w_n = {\bf 1}^T{\bf w}$ \\
subject to & $B'^{ T}{\bf w} \leq {\bf 1}$ \emph{and}  ${\bf w} \geq {\bf 0}$
\end{tabular}

\noindent In particular, the fractional chromatic number $\chi^*(H)$ can also be computed as the optimal value of the second linear program.

\end{lemma} 

\begin{proof}
It is clear from the definitions that there is a block decomposition 
\[ 
B =
\begin{bmatrix}
B' & C
\end{bmatrix}
\]
 where $C$ is a $n\times (t-s)$ matrix corresponding to non-maximal independent sets. The feasible set for the first LP is thus given by the constraints $B'^{ T}{\bf w} \leq {\bf 1}$, $C^T{\bf w} \leq {\bf 1}$ and ${\bf w} \geq {\bf 0}$. It is clear that any feasible solution of the first LP is also feasible for the second. For the converse we need to observe that the constraint equations $C^T{\bf w} \leq {\bf 1}$ are all
redundant. To see why, note that any row in $C^T$ corresponds to a
non-maximal independent set $W'$.  So, there is a maximal
independent set $W$ such that $W' \subseteq W$, and if  ${\bf w}$ satisfies the constraint corresponding to the row $W$, it will also have to satisfy the constraint coming from
the row corresponding to $W'$. In particular, this tells us that
$B'^{\ T}{\bf w} \leq {\bf 1}$ implies $C^T{\bf w} \leq {\bf 1}$, and
consequently the two LPs have the same feasible sets. Since the LPs also have the same objective function, their optimal values will be the same. Since the first LP is the dual of the LP in Definition \ref{fracchromdefn}, the common value of these LPs is equal to $\chi^*(H)$.
\end{proof}

Our goal is now to show that if $I$ is any squarefree monomial ideal
of $R$, and if $H$ is the hypergraph such that $I = I(H)$,
then $\widehat\alpha(I)$ can be expressed in terms of $\chi^*(H)$.
To do this, we relate the matrix $A$
with the matrix
$B'$ of Lemma \ref{maxindeplemma}.

\begin{lemma}\label{matrixformulation}
Let $H = (V,E)$ be a hypergraph with edge ideal $I = I(H)$.
Suppose that $W = \{W_1,\ldots,W_s\}$ are the maximal independent
sets of $H$. 
Let $A$ be the $s \times n$ matrix of the LP of Theorem \ref{mainresult}
constructed from $I(H)$, and let $B'$ be the $n \times s$ matrix of
Lemma \ref{maxindeplemma}.  Then
\[ 
B' =(\mathbb{I}-A)^T \mbox { and } A =(\mathbb{I}-B')^T
\]
where $\mathbb{I}$ denotes an appropriate sized matrix with every entry equal to one.
\end{lemma}

\begin{proof}
By Lemma \ref{vertexcovers},
a set of variables generates a minimal prime ideal
containing $I$ if and only if its complement is a maximal independent set of $H$, i.e., there is a one-to-one
correspondence between the associated primes $P_1,\ldots,P_s$
of $I(H)$ and the maximal independent sets of $H$.
This complementing is represented by the formula $\mathbb{I}-A$ or $\mathbb{I}-B'$, while transposition
occurs since the variables index rows for $B'$ and columns for $A$.
\end{proof}

\begin{theorem}\label{mainresult2}
Suppose that $H = (V,E)$ is a hypergraph with a non-trivial edge,
and let $I = I(H)$.  Then
$$\widehat\alpha(I)=\frac{\chi^*(H)}{\chi^*(H)-1}.$$
\end{theorem}

\begin{proof}
Consider  LP introduced in Lemma \ref{maxindeplemma}, namely

\begin{tabular}{rl}
maximize & $w_1 + \cdots + w_n = {\bf 1}^T{\bf w}$ \\
subject to & $B'^T{\bf w} \leq {\bf 1}$ \emph{and}  ${\bf w} \geq {\bf 0}$.
\end{tabular}

By Lemma \ref{maxindeplemma}, the
optimal value of this LP is  $\chi^*(H)$.
Let ${\bf w}^*$ denote an optimal solution for this LP.
We claim that $\frac{1}{\chi^*(H)-1}{\bf w}^*$ is a feasible solution
for the LP defining $\widehat\alpha(I)$. Indeed, using Lemma
\ref{matrixformulation}, we have
\begin{eqnarray*}
{\bf 1} & = & \frac{1}{\chi^*(H)-1} (\chi^*(H)\mathbf{1}-\mathbf{1})
 \leq   \frac{1}{\chi^*(H)-1} (\mathbb{I}{\bf w}^*-B'^T{\bf w}^*)\\
& = &
\frac{1}{\chi^*(H)-1} (\mathbb{I}-B')^T{\bf w}^*
=  \frac{1}{\chi^*(H)-1}A{\bf w}^*.
\end{eqnarray*}
In particular, ${\bf 1} \leq A\left(\frac{1}{\chi^*(H)-1}{\bf
w}^*\right)$, where ${\bf 1}$ is an appropriate sized vector of
$1$'s. Thus $$\widehat\alpha(I)\leq \frac{1}{\chi^*(H)-1}(w_1 +
\cdots + w_n)= \frac{\chi^*(H)}{\chi^*(H)-1}.$$

A similar computation shows that, if ${\bf y}^*$
is the optimal solution for the LP

\begin{tabular}{rl}
minimize & $y_1 + \cdots + y_n = {\bf 1}^T{\bf y}$ \\
subject to & $A{\bf y} \geq {\bf 1}$ and  ${\bf y} \geq {\bf 0}$
\end{tabular}

\noindent
that is, ${\bf 1}^T{\bf y}^* = \widehat\alpha(I)$,
then $\frac{1}{\widehat\alpha(I)-1}{\bf y}^*$ is a feasible solution
for the dual LP described in the beginning of this proof.
Indeed, using Lemma \ref{matrixformulation} we have
\begin{eqnarray*}
B'^T\left( \frac{1}{\widehat\alpha(I)-1}{\bf y}^* \right) & = &
(\mathbb{I}-A)\left( \frac{1}{\widehat\alpha(I)-1}{\bf y}^* \right) \\
& = & \frac{1}{\widehat\alpha(I)-1}(\mathbb{I} - A){\bf y}^*=\frac{1}{\widehat\alpha(I)-1}
\left(\widehat\alpha(I){\bf 1} - A{\bf y}^*\right).
\end{eqnarray*}
Because $A{\bf y}^* \geq {\bf 1}$, we now have
$
B'^T\left( \frac{1}{\widehat\alpha(I)-1}{\bf y}^* \right) \leq  {\bf 1}.
$
Thus Lemma \ref{LPbounds} yields the inequality
 $$\chi^*(H)\geq \frac{1}{\widehat\alpha(I)-1}(y_1 + \cdots + y_n)
=\frac{\widehat\alpha(I)}{\widehat\alpha(I)-1}$$ and by elementary manipulations this inequality
is equivalent to $\widehat\alpha(I)\geq\frac{\chi^*(H)}{\chi^*(H)-1}$.
\end{proof}

We end this section with an application that illustrates the power of Theorem \ref{mainresult2}. 

\begin{corollary}\label{disjoint}
Suppose that $I$ and $J$ are two squarefree monomial ideals
of the ring $R = k[x_1,\ldots,x_n,y_1,\ldots,y_n]$.
Furthermore, suppose that $I$ is generated
by monomials only in the $x_i$'s and
$J$ is generated by monomials only in the $y_j$'s.
Then
\[\widehat\alpha(I+J) = \min\{\widehat\alpha(I),\widehat\alpha(J)\}.\]
\end{corollary}

\begin{proof}
We can view $I$ as the edge ideal of a hypergraph $H$ on
the vertices $\{x_1,\ldots,x_n\}$ and $J$ as the edge
ideal of a hypergraph $K$ on the vertices $\{y_1,\ldots,y_n\}$.
Thus $I+J$ is the edge ideal of the hypergraph $H \cup K$ where
$H$ and $K$ are disjoint.  But then
\[\chi^*(H \cup K) = \max\{\chi^*(H),\chi^*(K)\},\]
which is equivalent to the statement
\[\frac{\chi^*(H \cup K)}{\chi^*(H \cup K) -1}
= \min\left\{
\frac{\chi^*(H)}{\chi^*(H)-1},\frac{\chi^*(K)}{\chi^*(K)-1}
\right\}.\]
Now apply Theorem \ref{mainresult2}.
\end{proof}


\section{A Chudnovsky-like lower bound on $\widehat\alpha(I)$}

Chudnovsky \cite{C} first proposed a conjectured lower bound on
$\widehat\alpha(I)$ when $I$ is the ideal of a set of points in a
projective space.  Motivated by this conjecture, Cooper, et al.
\cite{CEHH} formulated an analogous conjecture for all monomial
ideals.  Recall that the big height of $I$, denoted big-height$(I)$, is the maximum of the
heights of $P \in {\rm Ass}(I)$.

\begin{conjecture}[{\cite[Conjecture 6.6]{CEHH}}]
Let $I$ be a monomial ideal with big-height$(I) = e$. Then
\[\widehat\alpha(I) \geq \frac{\alpha(I)+e-1}{e}.\]
\end{conjecture}

\begin{remark}
In the original formulation, the authors make a conjecture about
$\alpha(\mathcal{Q})$ of
the symbolic polyhedron $\mathcal{Q}$ of $I$ as introduced
in Remark \ref{alpharemark}.  It is enough to know that in our
context, $\alpha(\mathcal{Q}) = \widehat\alpha(I)$.
\end{remark}

By taking the viewpoint that $\widehat\alpha(I)$ is the solution
to a LP, we are able to verify the above conjecture for all
squarefree monomial ideals.

\begin{theorem}\label{mainresult3}
Let $I$ be a squarefree monomial ideal with big-height$(I) = e$.  Then
\[\widehat\alpha(I) \geq \frac{\alpha(I)+e-1}{e}.\]
\end{theorem}

\begin{proof}
By Theorem \ref{mainresult}, $\widehat\alpha(I)$ is the optimum value of the
LP  that asks to minimize $y_1 + \cdots + y_n$ subject to the constraints
$A{\bf y} \geq {\bf 1}$ and ${\bf y} \geq {\bf 0}$, with $A$  obtained
from the primary decomposition of $I$. It is enough to show that any
feasible solution ${\bf y}$ for this LP satisfies
$$\sum_{i=1}^n y_i\geq \frac{\alpha(I)+e-1}{e}$$
in order to conclude that the optimal solution satisfies the same inequality,
hence the optimal value of the program satisfies the desired inequality
$\widehat\alpha(I)\geq \frac{\alpha(I)+e-1}{e}$.

Let $I=P_1\cap P_2\cap \cdots \cap P_s$ be the primary decomposition for $I$,
where the $P_i$ are prime ideals generated by a subset of the variables.
Since   $P_1P_2\cdots P_s\subseteq I$, we must have
$\alpha(P_1P_2\cdots P_s)\geq \alpha(I)$, hence $s\geq \alpha(I)$. The
feasible set of the above LP  is thus defined by at least $\alpha(I)$
inequalities. Since big-height$(I) = e$, each of these inequalities
involves at most $e$ of the variables. Both of these observations will be
used in the proof.

Let ${\bf y}$ be a feasible solution for the above LP.   If $\alpha(I) = 1$,
then because any constraint equation implies $y_1+\cdots+y_n
\geq 1$, the inequality $\sum_{i=1}^n {y_i} \geq \frac{\alpha(I)+e-1}{e}
= 1$ is satisfied.  So, we can assume that $\alpha(I) \geq 2$.

We will show that
there exist distinct indices $k_1,\ldots, k_{\alpha(I)-1}$ so that
$y_{k_i}\geq \frac{1}{e}$ for $1\leq i\leq  \alpha(I)-1$. The proof of this
claim is by induction.  For the base case, we need to find one index
$k_1$ such that $y_{k_1} \geq \frac{1}{e}$.  Let
$y_{i_1}+\cdots+y_{i_e} \geq 1$ be the constraint equation constructed
from the height $e$ associated prime.  Since ${\bf y}$ is a feasible
solution, at least one of $y_{i_1},\ldots,y_{i_e}$ must be $\geq \frac{1}{e}$.
Let $k_1$ be the corresponding index.
This proves our base case.

Now let $1< j\leq \alpha(I)-1$ and suppose that
there exist pairwise distinct indices $k_1,\ldots, k_{j-1}$ so that
$y_{k_i}\geq \frac{1}{e}$ for $1\leq i\leq  j-1$. Note that the monomial
$x_{k_1}x_{k_2}\cdots x_{k_{j-1}}$ of degree $j-1\leq \alpha(I)-2$ is not an
element of $I$. Consequently there exists a prime $P_\ell$ among the associated
primes of $I$ that does not contain the monomial $x_{k_1}x_{k_2}\cdots x_{k_{j-1}}$,
thus $P_\ell$ contains none of the variables $x_{k_1}, x_{k_2}, \ldots, x_{k_{j-1}}$.
Consider the inequality of the LP  corresponding to the prime $P_\ell$
$$y_{\ell_1}+y_{\ell_2}+\cdots+y_{\ell_{s_\ell}}\geq 1.$$
This inequality involves at most $e$ of the entries of ${\bf y}$, thus
$y_{\ell_t}\geq\frac{1}{e}$ for some $t$. Since $x_{\ell_t}\in P_\ell$ and none
of the variables $x_{k_1}, x_{k_2}, \ldots x_{k_{j-1}}$ are in $P_\ell$, we conclude
that $\ell_t$ must be distinct from any of the indices $k_1,\ldots, k_{j-1}$.
Setting $k_j=\ell_t$ gives a pairwise distinct set of indices
$k_1,\ldots, k_{j}$ so that  $y_{k_i}\geq \frac{1}{e}$ for $1\leq i\leq  j$.
This finishes the proof of our claim.

Now consider the monomial  $x_{k_1}x_{k_2}\cdots x_{k_{\alpha(I)-1}}$, which has
degree $\alpha(I)-1$ and consequently is not an element of $I$. Then there
exists an associated prime $P_u$ of $I$ so that none of the variables
$x_{k_1,}x_{k_2}, \ldots, x_{k_{\alpha(I)-1}}$ are in $P_u$. The inequality in the LP
corresponding to the prime $P_u$
$$y_{u_1}+y_{u_2}+\cdots+y_{u_{s_u}}\geq 1$$
together with the previously established inequalities
$$y_{k_1}\geq \frac{1}{e}, \  y_{k_2}\geq \frac{1}{e}, \ \ldots,
\ y_{k_{\alpha(I)-1}}\geq \frac{1}{e}$$
and the non-negativity conditions $y_i\geq 0$ for $1\leq i\leq n$ yield
\begin{eqnarray*}
\sum_{i=1}^n y_i &\geq& y_{k_1}+y_{k_2}+\ldots + y_{k_{\alpha(I)-1}}+
y_{u_1}+y_{u_2}+\ldots+y_{u_{s_u}}\\
&\geq & \frac{\alpha(I)-1}{e}+1=\frac{\alpha(I)+e-1}{e}.
\end{eqnarray*}
The first inequality also uses the fact that $\{k_1,k_2, \ldots,
k_{\alpha(I)-1}\} \cap \{u_1,u_2, \ldots, u_{s_u}\}=\varnothing$.

Since $\widehat\alpha(I)= \sum_{i=1}^n y^*_i$ for some feasible solution
${\bf y^*}$ of the LP, we now have
$$\widehat\alpha(I)= \sum_{i=1}^n y^*_i\geq \frac{\alpha(I)+e-1}{e}.$$
\end{proof}

\begin{remark}
The lower bound in the above theorem is optimal;  see
Theorem \ref{prop:Waldschmidt matroid} and Remark \ref{lowerboundremark}.
\end{remark}


\section{The Waldschmidt constant for edge ideals}\label{sec:results2}

In this section, we apply our methods to examine the Waldschmidt constant 
for edge ideals for several families of finite simple graphs, and  relate 
this algebraic invariant to invariants of the graph.

In the following, let $G=(V,E)$ be a finite simple graph with vertex set
$V=\{x_1,\ldots, x_n\}$ and edge set $E$. Let $k$ be a field and set
$R=k[x_1,\ldots x_n]$.  The {\it edge ideal} of $G$ is then the 
squarefree quadratic monomial ideal
\[I(G)= \langle x_ix_j \mid \{x_i,x_j\}\in E \rangle \subseteq R,\]
i.e., this is the special case of an edge ideal first introduced in Section 2. 
All terminology in that section can therefore be applied to graphs.
In particular, the notion of {\it vertex cover} specializes to graphs 
as well as the correspondence outlined in 
Lemma \ref{vertexcovers} which gives a bijection between minimal associated primes 
of $I(G)$ and minimal vertex covers of $G$. 



\begin{definition}
A $k$-colouring for $G$ is an assignment of $k$ labels (or colours) to the elements of $V$
so that no two adjacent vertices are given the same label.
The \emph{chromatic number} of $G$, $\chi(G)$, is the smallest
integer $k$ so that $G$ admits a $k$-colouring.
\end{definition}

\begin{definition}
A  \emph{clique} of $G$ is a set of pairwise adjacent vertices of $G$.
A maximum clique of $G$ is a clique such that $G$ admits no clique with
more vertices.
The clique number $\omega(G)$  is the number of vertices in a maximum
clique in $G$.
\end{definition}

We obtain the following bound on 
$\widehat\alpha(I(G))$ in terms of these invariants.

\begin{theorem}\label{chromatic}
Let $G$ be a non-empty graph with chromatic number $\chi(G)$ and clique
number $\omega(G)$.
Then  
$$\frac{\chi(G)}{\chi(G)-1}\leq \widehat\alpha(I(G)) \leq 
\frac{\omega(G)}{\omega(G)-1}.$$
\end{theorem}

\begin{proof}
The fractional chromatic number $\chi^*(G)$ of the graph $G$ 
is the solution to the LP of Defintion \ref{fracchromdefn}.  Now
$\chi(G)$ is the integer solution to this LP, while $\omega(G)$ is the 
integer solution to the dual of this LP.  This
implies that  $\omega(G) \leq \chi^*(G) \leq \chi(G)$, and so the
result follows from
Theorem \ref{mainresult2} which gives 
$\widehat\alpha(I(G)) = \frac{\chi^*(G)}{\chi^*(G)-1}$.
\end{proof}

The above lower bound improves the lower bound from Theorem \ref{mainresult3}.

\begin{theorem}Let $I(G)$ be the edge ideal of a graph $G$ 
and let big-height$(I(G)) = e$.  Then
\[
\widehat\alpha(I(G)) \geq \frac{\chi(G)}{\chi(G)-1} \geq \frac{e+1}{e}
= \frac{\alpha(I(G))+e-1}{e}.\]
\end{theorem}

\begin{proof}\label{improvedlowerbound}
Theorem \ref{chromatic} already shows the first inequality,
so it suffices to verify the second inequality
$\chi(G)/(\chi(G)-1) \geq (e+1)/e.$

Let $P$ be the associated prime of $I(G)$ with
height $e$.  If $P = \langle x_{i_1},\ldots,x_{i_e}\rangle$,
then $W = \{x_1,\ldots,x_n\} \setminus \{ x_{i_1},\ldots,x_{i_e}\}$
is an independent set of $G$.  We can now colour $G$ with $e+1$ colours
by colouring the vertices of $W$ one colour, and then colour 
each vertex of $\{ x_{i_1},\ldots,x_{i_e}\}$ with a distinct colour.
So $\chi(G) \leq e+1$, which gives the result.
\end{proof}

We now turn to the computation of the Waldschmidt constant for 
various families of simple graphs. In particular, we examine 
perfect graphs, $k$-partite graphs, cycles, and complements of cycles.
We will use these results to give a simplified proof to a result of Bocci and
Franci \cite{BF,F}.

We now recall the definitions of the family of graphs we wish to study.
If $G = (V,E)$ is a graph and $A \subseteq V$, then the {\it induced subgraph
of $G$ on $A$}, denoted $G_A$, is the graph $G_A = (A,E_A)$ where 
$E_A = \{e \in E ~|~ e \subseteq A\}$.  We say a graph $G$ is {\it perfect}
if $\omega(G_A) = \chi(G_A)$ for all $A \subseteq V$.  A graph $G = (V,E)$
is a {\it $k$-partite} graph if there exists a $k$-paritition
$V = V_1 \cup \cdots \cup V_k$ such that no $e \subseteq V_i$ for any $i$.
When $k =2$, we call $G$ {\it bipartite}.  The {\it complete $k$-partite} graph
is a graph with $k$-paritition $V = V_1 \cup \cdots \cup V_k$ 
and all edges of the form $\{v_i,v_j\}$ with $v_i \in V_i$ and $v_j \in V_j$
and $i \neq j$.
 The {\it complete
graph} on $n$ vertices, denoted $K_n$, is the graph on the vertex
set $V = \{x_1,\ldots,x_n\}$ and edge set 
$\{\{x_i,x_j\} ~|~ 1 \leq i < j \leq n\}$.  The {\it cycle} on $n$
vertices, denoted $C_n$, is a graph on $V = \{x_1,\ldots,x_n\}$ and
edge set $\{\{x_1,x_2\},\{x_2,x_3\},\ldots,\{x_{n-1},x_n\},\{x_n,x_1\}\}$.  
The {\it complement} of a graph $G= (V,E)$, denoted $G^c$, is the graph
with the same vertex set as $G$, but edge set $\{\{x_i,x_j\} ~|~ 
\{x_i,x_j\} \not\in E\}$.

We will use the following result to compute (or bound) $\chi^*(G)$.

\begin{definition}
A graph $G$ is {\it vertex-transitive} if for all 
$u,v\in V(G)$ there is an automorphism $\pi$ of $G$ with $\pi(u)=v$.
\end{definition}

\begin{theorem}[{\cite[Proposition 3.1.1]{SU}}]\label{vertextransitive}
If $G$ is any graph, then $\chi^*(G) \geq \frac{|V(G)|}{\alpha(G)}$, 
where $\alpha(G)$ is the independence number of $G$ 
(i.e. the size of the largest independent set in $G$).
Equality holds if $G$ is vertex-transitive.
\end{theorem}

Examples of vertex-transitive graphs are complete graphs, cycles,
and their complements.    We are now able to compute
$\widehat\alpha(I(G))$ for a large number of families of graphs.

\begin{theorem}\label{thm:graphfamilies}
Let $G$ be a non-empty graph.

\begin{enumerate}
\item[$(i)$] \label{perfectgraph} 
If $\chi(G) = \omega(G)$, 
then $\widehat\alpha(I(G)) = \frac{\chi(G)}{\chi(G)-1}$. 
In particular, this equality holds for all perfect graphs.
\item[$(ii)$] 
\label{kpartite} If $G$ is $k$-partite, 
then $\widehat\alpha(I(G)) \geq \frac{k}{k-1}$.
In particular, if $G$ is a complete $k$-partite graph, 
then $\widehat\alpha(I(G)) = \frac{k}{k-1}$.
\item[$(iii)$]  If $G$ is bipartite, then $\widehat\alpha(I(G)) = 2$.
\item[$(iv)$] 
If $G = C_{2n+1}$ is an odd cycle, 
then $\widehat\alpha(I(C_{2n+1})) = \frac{2n+1}{n+1}$.
\item[$(v)$]\label{oddcomplement} 
If $G = C^c_{2n+1}$, then $\widehat\alpha(I(G)) = \frac{2n+1}{2n-1}$.
\end{enumerate}
\end{theorem}

\begin{proof}
$(i)$ This result follows immediately from Theorem \ref{chromatic}.  Note
that perfect graphs have the property that $\omega(G) = \chi(G)$.

$(ii)$.
If $G$ is a $k$-partite graph, then $\chi(G)\leq k$; indeed,
if $V = V_1 \cup \cdots \cup V_k$ is the $k$-parititon, colouring
all the vertices of $V_i$ the same colour gives a valid colouring.  
By Theorem
\ref{chromatic},
$\widehat\alpha(I(G)) \geq \frac{\chi(G)}{\chi(G)-1} \geq  \frac{k}{k-1}$.
If $G$ is a complete $k$-partite graph, then
$\chi(G)\leq k= \omega(G)$ and the desired equality follows by a direct
application of Theorem \ref{chromatic}.

$(iii)$ For any bipartite graph $G$, $\chi(G)=\omega(G)=2$, so apply $(i)$.

$(iv)$ For an odd cycle $C_{2n+1}$, $\chi^*(C_{2n+1}) = 2+1/n$ 
by Theorem \ref{vertextransitive}.  Now apply Theorem \ref{mainresult2}.

$(v)$ For the complement $G$ of $C_{2n+1}$, $\chi^*(G) = \frac{2n+1}{2}$ 
by Theorem \ref{vertextransitive}.   Again, apply Theorem \ref{mainresult2}.
\end{proof}

\begin{remark}
The fact that $\widehat\alpha(I(G)) = 2$ when $G$ is bipartite is well-known.  
In fact, the much stronger result that $I(G)^{(m)} = I(G)^m$
for all $m$ holds when $G$ is bipartite (see \cite{SVV}).
\end{remark}

Bocci and Franci \cite{BF} recently computed the Waldschmidt constant 
of the Stanley-Reisner ideal of the so-called $n$-bipyramid.  We illustrate 
the strength of our new techniques by giving a simplified proof of 
their main result using the above results.

\begin{definition}
The {\it bipyramid over a polytope $P$}, denoted $\bipyr(P)$, is the 
convex hull of $P$ and any line segment which meets the interior of $P$
at exactly once point.
\end{definition}

Bocci and Franci considered the bipyramid of an $n$-gon.  Specifically,
let $Q_n$ be an $n$-gon in $\mathbb{R}^2$, with 
vertices $\{1,\ldots,n\}$, containing the origin and embedded in 
$\mathbb{R}^3$. We denote by $B_n$ the bipyramid over $Q_n$, i.e., 
the convex hull
\[B_n = {\rm bipyr}(Q_n) = \conv(Q_n,(0,0,1),(0,0,-1)).\]

For a simplicial complex $\Delta$ with vertices 
$\{1,\ldots,n\}$, we may identify a subset 
$\sigma\subseteq \{1,\ldots,n\}$ with the $n$-tuple in 
$\{0,1\}^n$ and we adopt the convention that 
$x^\sigma = \prod_{i\in \sigma} x_i$.
The Stanley-Reisner ideal of a simplicial complex 
$\Delta$ on vertices $\{1,\ldots,n\}$ is defined to be
$
I_\Delta = \langle x^\sigma \mid \sigma\notin \Delta \rangle,
$
i.e., it is generated by the non-faces of $\Delta$.

We view $B_n$ as a simplicial complex on the vertex set 
$\{x_1,\ldots,x_n,y,z\}$ where the $x_i$'s correspond to the vertices
of the $n$-gon, and $y$ and $z$ correspond to the end points of the
line segment that meets the interior of the $n$-gon at one point.
Because the bipyramid $B_n$ is a simplicial complex, we let $I_n = I_{B_n}$
be the Stanley-Reisner ideal associated to $B_n$.  Bocci and
Franci \cite[Proposition 3.1]{BF} 
described the generators of $I_n$;  in particular,
\begin{equation}\label{Ingen}
I_n = \langle yz \rangle + \langle x_ix_j ~|~ \mbox{$i$ and $j$
non-adjacent in $Q_n$} \rangle.
\end{equation}
Note that $I_n$ can be viewed as the edge ideal of some graph since
all the generators are quadratic squarefree monomials.  Using 
the results of this section, we have shown: 

\begin{theorem}[{\cite[Theorem 1.1]{BF}}]
Let $I_n$ be the Stanley-Reisner ideal of the $n$-bipyramid $B_n$.   
Then  $\widehat\alpha(I_n) = \frac{n}{n-2}$ for all $n \geq 4$.
\end{theorem}

\begin{proof}
The ideal $I_n$ is an ideal in the polynomial ring $R = k[x_1,\ldots,x_n,y,z]$.
By \eqref{Ingen} $I_n$ can be viewed as the edge
ideal of the graph $G_n$ where $G_n = H \cup C_n^c$ consists of two
disjoint components.  In particular, $H$ is the graph of a single edge
$\{y,z\}$ and $C_n^c$ is the complement of the $n$-cycle $C_n$.  
By Corollary \ref{disjoint}
 to compute $\widehat\alpha(I_n)$
it suffices to compute $\chi^*(G_n) = \max\{\chi^*(H),\chi^*(C_n^c)\}$.  A graph
consisting of a single edge is perfect, so $\chi^*(H) = 2$.  On
the other hand,
\[\chi^*(C_n^c) =
\begin{cases}
m & \mbox{if $n =2m$} \\
m+ \frac{1}{2} & \mbox{if $n =2m+1$}
\end{cases}\]
So, if $n > 3$, $\chi^*(G_n) = \chi^*(C_n^c)$.

Thus, if $n =2m$, $\widehat\alpha(I_n) = \frac{m}{m-1} = \frac{n}{n-2}$.  And
if $n = 2m+1$, then
\[\widehat\alpha(I_n) = \frac{m+\frac{1}{2}}{m-\frac{1}{2}} = \frac{n}{n-2}.\]
In other words, $\widehat\alpha(I_n) = \frac{n}{n-2}$ for all $n \geq 4$.
\end{proof}

We close this section with some comments about the
Alexander dual.

\begin{definition}
Let $I = P_1 \cap \cdots \cap P_s$ be a squarefree monomial ideal
with $P_i = \langle x_{j_1},\ldots,x_{j_{s_j}} \rangle$ for
$j=1,\ldots,s$.   Then the {\it Alexander
dual} of $I$, denoted $I^\vee$, is the monomial ideal
$I^\vee = \langle x_{j_1}\cdots x_{j_{s_j}}~|~
j=1,\ldots,s\, \rangle.$
\end{definition}

In combinatorial commutative algebra, the Alexander dual of a monomial
ideal $I$ is used quite frequently to deduce additional information about
$I$.  It is thus natural to ask if knowing $\widehat\alpha(I)$ of a 
squarefree monomial ideal allows us to deduce any information 
about $\widehat\alpha(I^\vee)$.  As
the next example shows,  simply knowing $\widehat\alpha(I)$ gives no 
information on $\widehat\alpha(I^\vee)$.

\begin{example}\label{dual}
Let $s \geq 1$ be an integer, and let $G_s = K_{s,s}$ be the complete bipartite
graph on the vertex set $V = \{x_1,\ldots,x_s\} \cup \{y_1,\ldots,y_s\}$.
Now $\widehat\alpha(I(G_s)) = 2$ by 
Theorem \ref{thm:graphfamilies} for all $s \geq 1$.
On the other hand, since 
$I(G_s) = \langle x_1,\ldots,x_s \rangle \cap \langle y_1,\ldots,y_s \rangle,$
we have
\[I(G_s)^\vee = \langle x_1\cdots x_s, y_1 \cdots y_s \rangle.\]
But the ideal $I(G_s)^{\vee}$ is a complete intersection so
$(I(G_s)^{\vee})^{(m)} = (I(G_s)^{\vee})^m$ for all $m$.  In particular,
$\alpha((I(G_s)^{\vee})^{(m)} = \alpha((I(G_s)^{\vee})^{m} = sm$.   So
$\widehat\alpha(I(G_s)^\vee) = s$.

We see that if we only know that $\widehat\alpha(I) = 2$,
then $\widehat\alpha(I^\vee)$ can be any positive integer.  We require
further information about $I$ to deduce any information about 
$\widehat\alpha(I^\vee)$.
\end{example}


\section{Some applications to the ideal containment problem}

As mentioned in the introduction, the renewed interest in the Waldschmidt constant
grew out of the activity surrounding the containment problem for ideals of 
subschemes $X$ of $\mathbb{P}^n$, i.e., determine all positive integer pairs 
$(m,r)$ such that $I^{(m)}\subseteq I^r$ where $I = I(X)$.  We apply our 
technique for computing $\widehat\alpha(I)$
to examine the containment problem for three families  of monomial ideals:
(1) a union of a small number (when compared to $n$) 
of general linear varieties, 
(2) the Stanley-Reisner ideal
of a uniform matroid, and (3)  a family of monomial ideals 
of mixed height. 
Note that for this section, we shall assume that 
$R = k[\mathbb{P}^n] = k[x_0,\ldots,x_n]$.

Before turning to our applications, we recall some relevant background.
To study the containment problem, Bocci and Harbourne \cite{BH} introduce 
the {\it resurgence} of $I$, that is,
$$\rho(I) = \sup\left\{\frac{m}{r} ~|~ I^{(m)} \not\subseteq I^r \right \}.$$
An asymptotic version of resurgence was later defined by 
Guardo, Harbourne, and Van Tuyl \cite{GHVT}
as
\[\rho_a(I) = 
\sup\left.\left\{\frac{m}{r} ~\right|~ I^{(mt)} \not\subseteq I^{rt} \, \, \, 
\mbox{for all $t \gg 0$.}\right\}\]
These invariants are related to the Waldschmidt constant of $I$  as follows.

\begin{lemma}[{\cite[Theorem 1.2]{GHVT}}]\label{resurgencebounds}
Let $I \subseteq R = k[x_0,\ldots,x_n]$ be a homogeneous ideal.  Then
\begin{enumerate}
\item[$(i)$] $1 \leq \alpha(I)/\widehat\alpha(I) \leq \rho_a(I) \leq \rho(I)$.
\item[$(ii)$] If $I$ is the ideal of a smooth subscheme of $\mathbb{P}^n$, then
$\rho_a(I) \leq \omega(I)/\widehat\alpha(I)$
where $\omega(I)$ denotes the largest degree of a minimal generator.
\end{enumerate}
\end{lemma}


\subsection{Unions of general linear varieties}
In \cite[Theorem 1.5]{GHVT}, the values of 
$\widehat\alpha(I)$ and $\rho (I)$  are established when $I$ is the ideal of certain linear 
subschemes of $\mathbb{P}^n$ in general position.  The key idea is that when
the number of linear varieties is small, we can assume that the defining
ideal of $I$ is a monomial ideal.
By using Theorem \ref{mainresult} to compute the Waldschmidt constant,
we are able to recover and extend the original result.

\begin{theorem}\label{generallinear}
Let $X$ be the union of $s$ general linear subvarieties 
$L_1,\ldots,L_s$, each of dimension $t-1$. Assume 
 $st \leq n+1$ and set $I = I(X)$. Then
\[
\widehat\alpha(I) = \begin{cases}
  1 & \text{if $1\leq st < n+1$}\\
  \frac{n+1}{n+1-t} & \text{if $st = n+1$.}
\end{cases}
\]
Additionally, if  $s \ge 2$, then the resurgences are 
\[
\rho (I) = \rho_a (I) = \frac{2 \cdot (s-1)}{s}. 
\]
Furthermore,
\begin{eqnarray*}
 \frac{\alpha(I)}{\widehat\alpha(I)} =  \rho_a (I) = \rho(I) = \frac{\omega (I)}{\widehat\alpha(I)}, & \text{if } n +1 = st\\
 \frac{\alpha(I)}{\widehat\alpha(I)} < \rho_a (I) = \rho(I) < \frac{\omega (I)}{\widehat\alpha(I)},   & \text{if } s t \le n \text{ and } s \ge 3.
 \end{eqnarray*} 
%
\end{theorem}

\begin{remark}\label{nonequal1}
If $s=1$, then the ideal $I$ of Theorem \ref{generallinear} 
is generated by variables, and so 
$\rho (I) = \rho_a (I)= 1$. Thus, the assumption $s \ge 2$ is harmless. 
The case $t=1$ in Theorem \ref{generallinear}
was first proved in \cite{BH}, while the case 
$t=2$ is found in \cite{GHVT}. 

The final assertion of Theorem \ref{generallinear} gives examples where neither the lower bound nor the upper bound for the asymptotic resurgence in Lemma \ref{resurgencebounds} are sharp. 
\end{remark}

As preparation, we note the following observation. 

\begin{lemma}
\label{lem:regular element}
Let $0 \neq I \subset R$ be a monomial ideal, and let $y$ be a new variable. 
Consider the ideal $(I, y)$ in $S = R[y]$. Then 
\[
\rho\left((I, y)\right) = \rho (I) \quad \text{ and } \quad \rho_a\left((I, y)\right) = \rho_a (I). 
\]
\end{lemma}

\begin{proof} 
First we show $\rho((I, y)) \ge  \rho (I)$ and $\rho_a ((I, y)) \ge_a  \rho (I)$. To this end, assume 
$I^{(m t)} \nsubseteq  I^{r t}$ for some positive integers $m, r$ and $t$. Thus, 
there is a monomial $x^a = x_0^{a_0} \cdots x_n^{a_n} \in I^{(mt)}$ 
with $x^a \notin I^{rt}$. 
It follows that $x^a \in (I,y)^{(mt)}$, but $x^a \notin (I, y)^{rt}$, which implies both 
$\rho ((I, y)) \ge  \rho (I)$ and $\rho_a ((I, y)) \ge  \rho_a (I)$. 

Second, we prove $\rho ((I, y)) \le  \rho (I)$. Consider  positive integers $m$ and $r$ with $\frac{m}{r} > \rho(I)$. It 
suffices to show $(I,y)^{(m)} \subseteq (I,y)^r$. 
To this end consider a minimal generator $x^a y^b$ of $(I,y)^{(m)}$. 
If $b \ge r$, 
then clearly $x^a y^b \in (I, y)^r$, and we are done. Otherwise, $b < r < m$, 
and $x^a y^b \in (I,y)^{(m)}$ gives $x^a \in I^{(m-b)}$. Now 
$\frac{m-b}{r-b} \ge \frac{m}{r} > \rho (I)$ 
implies $I^{(m-b)} \subseteq I^{r-b}$, 
and hence $x^a \in I^{r-b}$. It follows that $x^a y^b$ is in $(I,y)^r$, 
which shows 
$(I,y)^{(m)} \subseteq (I,y)^r$. 

Similarly, one establishes $\rho_a ((I, y)) \le  \rho_a (I)$
\end{proof}

\begin{proof}[Proof of Theorem \ref{generallinear}]
Because the linear varieties $L_i$ are in general position, we may assume
\[
I(L_i) = (x_0,x_1,\ldots,\widehat x_{(i-1)t},
\ldots,\widehat x_{it-1},x_{it},\ldots,x_n),
\]
where the $~\widehat{}~$ denotes an omitted variable.
In particular, $I(X) = \bigcap_{i=1}^s I(L_i)$ is a squarefree monomial ideal,
so we can apply Theorem \ref{mainresult} to calculate $\widehat\alpha(I(X))$.

If $1 \leq st < n+1$, we wish to minimize $x_0 + x_1 + \cdots + x_n$
subject to
\[x_0 + x_1 + \cdots + \widehat x_{(i-1)t} + \cdots +
   \widehat x_{it-1} + \cdots + x_n \geq 1 ~~\mbox{for $i = 1,\ldots,s$.}\]
Since $st < n+1$, the vector 
$\mathbf{y}^T = \begin{bmatrix}0 & \cdots & 0 & 1 \end{bmatrix}$
is a feasible solution, so the minimum is at most 1.
On the other hand, because $x_0 + \cdots + x_n \geq x_t + \cdots + x_n \geq 1$,
the minimum solution is at least 1.    So $\widehat\alpha(I(X)) = 1$ in this situation.

If $st = n+1$, then the matrix $A$ of Theorem \ref{mainresult} is an $s\times (n+1)$ matrix whose
$i$-th row consists of $(i-1)t$ $1$'s, followed by $t$ $0$'s, followed by $n+1-it$ $1$'s.
The vector ${\bf y}$ with 
\[
\mathbf{y}^T = 
\underbrace{
\begin{bmatrix}
\frac{1}{n+1-t}& \cdots & \frac{1}{n+1-t}
\end{bmatrix}}_{n+1}
\]
is a feasible solution to the linear program of Theorem \ref{mainresult}, and so
$\widehat\alpha(I(X)) \leq \frac{n+1}{n+1-t}$.
The associated dual linear program is as follows: maximize $y_0 + \cdots + y_n$ such that
$A^T \mathbf{y} \leq \mathbf{1}$.
We claim that the $s$-tuple ${\bf y}$ with
\[
\mathbf{y}^T = 
\begin{bmatrix}\frac{t}{n+1-t} & \cdots & \frac{t}{n+1-t}
\end{bmatrix}
\]
is a feasible solution.  Indeed, observe that in each entry, 
$A^T \mathbf{y}$ is $(s-1)\left(\frac{t}{n+1-t}\right)
= \left(\frac{n+1-t}{t}\right)\left(\frac{t}{n+1-t}\right) = 1$, 
and thus $\widehat\alpha(I(X)) \geq \frac{st}{n+1-t} = 
\frac{n+1}{n+1-t}$.
Combining inequalities gives us the desired result 
for the Waldschmidt constant. 

For the remaining claims, we consider first the case where $n+1 = st$.  
Note that then $\alpha(I(X)) = \omega (I(X)) = 2$. 
Moreover, $s \ge 2$ implies $t \le \frac{n+1}{2}$. It follows that 
$2 (t-1) = \dim L_i + \dim L_j < n$, and thus $X$ is smooth.  
Hence, Lemma \ref{resurgencebounds}
gives \[\rho_a(I(X)) = \frac{\alpha(I(X))}{\widehat\alpha(I(X))} = \frac{2}{(n+1)/(n+1-t)} 
=  \frac{2 \cdot (s-1)}{s} \leq \rho (I(X)). 
\]
Thus, in order to determine $\rho (I(X))$
it suffices to show: 
If $m$ and $r$ are positive integers with 
$\frac{m}{r} > \frac{2 \cdot (s-1)}{s}$, 
then $I(X))^{(m)} \subset I(X)^r$. To this end we adapt the 
argument employed in the 
proof of \cite[Theorem 1.5]{GHVT}. Consider the ring homomorphism 
\[
\ffi: R \to S = k[y_0,\ldots,y_{s-1}],  \text{ defined by } x_i \mapsto y_j 
\text{ if } jt \leq i < (j+1) t. 
\]
Note that, for each $i \in [s]=  \{1,\ldots,s\}$, the ideal of 
$S$ generated by   
$\ffi (I(L_i))$ is $P_i = (y_0,\ldots,\widehat{y}_{i-1},\ldots,y_{s-1})$. Thus, 
$J = P_1 \cap \cdots \cap P_s$ is the ideal of the $s$ coordinate points in 
$\mathbb P^{s-1}$. 

Consider a monomial $x^a = x_0^{a_0} \cdots x_n^{a_n} \in R$. Then $x^a$ is in 
$I^{(m)} = \bigcap_{i=1}^s I(L_i)^m$ if and only if 
$\deg (x^a) - (a_{(i-1) t} + a_{(i-1) t + 1}  +  \cdots + a_{i t - 1}) \ge m$ 
for each  $i \in [s]$. Furthermore, 
a monomial $y^b = y_0^{b_0} \cdots y_{s-1}^{b_{s-1}} \in S$ is 
in $J^{(m)} = \bigcap_{i=1}^s P_i^m$ if and only 
if $\deg (y^b) - b_{i-1} \ge m$ for 
every $i \in [s]$. It follows that  
\begin{equation}
  \label{eq:monomial}
x^a \in I^{(m)} \; \text{ if and only if } \; \ffi (x^a) \in J^{(m)}. 
\end{equation} 

Consider now any monomial $x^a$ in $I^{(m)}$. The 
equivalence \eqref{eq:monomial} 
gives $\ffi (x^a) \in J^{(m)}$.  
Since $\rho (J) = \frac{2 \cdot (s-1)}{s}$ by \cite[Theorem 2.4.3]{BH}, our 
assumption $\frac{m}{r} > \frac{2 \cdot (s-1)}{s}$ yields 
$J^{(m)} \subseteq J^r$. 
Hence, we can write $\ffi (x^a) = \pi_1 \cdots \pi_r$, 
where each $\pi_j$ is a monomial in $J$. 
Equivalence \eqref{eq:monomial} implies now that there are monomials 
$\mu_1,\ldots,\mu_r \in I$ such that $\ffi (\mu_j) = \pi_j$ for each $j$ and 
$x^a = \mu_1 \cdots \mu_r$. It follows that $x^a \in I^r$, and hence  
$I^{(m)} \subset I^r$, as desired. 

Finally, assume $n \ge s t$. Then $I(X)$ is the sum of  
$n+1 - s t$ variables 
and the extension ideal $I(Y) R$ of  the ideal of the union $Y$ of $s$ 
general $(t-1)$-dimensional linear subspaces in $\mathbb{P}^{st-1}$. Thus, 
Lemma \ref{lem:regular element} yields $\rho (I(X)) = \rho (I(Y))$ and $\rho_a (I(X)) = \rho_a (I(Y))$, and hence 
$\rho (I(X)) = \rho_a (I(X)) = \frac{2 \cdot (s-1)}{s}$. However, $\frac{\alpha(I(X))}{\widehat\alpha(I(X))} = \frac{1}{1} = 1$ and $\frac{\omega (I(X))}{\widehat\alpha(I(X))}= \frac{2}{1} = 2$. 
\end{proof}


\subsection{Stanley-Reisner ideals of uniform matroids}
We use our methods to determine the Waldschmidt constant of 
the Stanley-Reisner ideal $I_{n+1,c}$ of a uniform matroid $\Delta$ on $n+1$ 
vertices whose facets are all the cardinality $n+1-c$ subsets 
of the vertex set.  These
ideals were also recently studied by Geramita, Harbourne, Migliore, and Nagel \cite{GHMN} and
Lampa-Baczy\'{n}ska and Malara \cite{LM}.
The ideal $I_{n+1,c}$ is generated by all squarefree monomials of 
degree $n+2-c$ in $R$. Equivalently,
\begin{equation}
\label{eq:matroid}
I_{n+1, c} = \bigcap_{0 \le i_1 < i_2 < \cdots < i_c \le n} (x_{i_1}, x_{i_2},\ldots,x_{i_c}).
\end{equation}

\begin{theorem}
 \label{prop:Waldschmidt matroid} The 
Stanley-Reisner ideal of a $(n-c)$-dimensional uniform matroid 
on $n+1$ vertices has Waldschmidt constant
\[
\widehat\alpha(I_{n+1, c}) = \frac{n+1}{c}.
\]
\end{theorem}

\begin{proof}
This follows by \cite[Lemma 2.4.1, Lemma 2.4.2 and the proof of Theorem 2.4.3]{BH}. 
However, we wish to provide a direct argument here.

We use Theorem \ref{mainresult}. By Equation \eqref{eq:matroid}, each 
left-hand side of the $\binom{s+1}{c}$ inequalities  in $A{\bf y}  \geq {\bf 1}$ 
is a sum of $c$ distinct variables in $\{y_1,\ldots,y_{n+1}\}$. Thus,
\[
y_1 = \cdots = y_{n+1} = \frac{1}{c}
\]
is a feasible solution, which gives 
$\widehat{\alpha} (I_{n+1, c}) \le \frac{n+1}{c}$.

Now observe that each $y_i$ appears in $\binom{n}{c-1}$ inequalities of
$A{\bf y}  \geq {\bf 1}$. Hence, summing over all these inequalities we get
\[
\binom{n}{c-1} [y_1 + \cdots + y_{n+1}] \ge \binom{n+1}{c},
\]
which yields $\widehat\alpha(I_{n+1, c}) \ge \binom{n+1}{c}/\binom{n}{c-1} = 
\frac{n+1}{c}$, completing the argument.
\end{proof}

\begin{remark}\label{lowerboundremark}
The ideal $I_{n+1,c}$ give an example of a family of squarefree 
monomial ideals that
achieves the lower bound of Theorem \ref{mainresult3}.   Indeed, 
big-height$(I_{n+1,c}) = c$ and $\alpha(I_{n+1,c}) = n+2-c$, so 
\[\widehat\alpha(I_{n+1,c}) = \frac{n+1}{c} = \frac{\alpha(I) + c - 1}{c}.\]
\end{remark}

\subsection{A binomial-like theorem for symbolic powers of 
monomial ideals} 

In this subsection we introduce some results which will be useful in Subsection \ref{mixedeight}. They are also
of independent interest.

\begin{lemma}\label{LemmaFactorSymbolicPowerProduct}
Let $I_1$ and $I_2$ be squarefree monomial ideals whose generators are
expressed in terms of disjoint sets of variables $x_1,\ldots,x_n$ and
$y_1,\ldots,y_m$, respectively.
Let ${\rm Ass}(I_1) = \{P_1,\ldots,P_a\}$ and ${\rm Ass}(I_2) = 
\{Q_1,\ldots,Q_b\}$.
Then
\[
    I_1^{(s)} I_2^{(t)} = \bigcap\limits_{i,j} P_i^s Q_j^t.
\]
\end{lemma}

\begin{proof} All of these ideals are monomial. Since $P_i^s$ and $Q_j^t$ are ideals written in terms of distinct sets of variables,  $P_i^s \cap Q_j^t=P_i^s Q_j^t$ and by the same reasoning  $I_1^{(s)}\cap I_2^{(t)} = I_1^{(s)} I_2^{(t)}$. Thus
$$\bigcap\limits_{i,j} P_i^s Q_j^t=\bigcap\limits_{i,j} (P_i^s\cap Q_j^t)=\left(\bigcap\limits_{i}P_i^s\right) \cap \left(\bigcap\limits_{j}Q_j^t\right)= I_1^{(s)}\cap I_2^{(t)} = I_1^{(s)} I_2^{(t)}.$$
\end{proof}

From here, we deduce the symbolic binomial theorem.

\begin{theorem}\label{LemmaSymbolicBinomialTheorem}
 Let $I_1$ and $I_2$ be squarefree monomial ideals in variables
$x_1,\ldots,x_n$ and $y_1,\ldots,y_m$, respectively.
 Set $I = I_1 + I_2$.
    Then
    \[
        I^{(m)} = \sum\limits_{j=0}^m I_1^{(m-j)} I_2^{(j)}.
    \]
\end{theorem}

\begin{proof}
Using the notation of Lemma \ref{LemmaFactorSymbolicPowerProduct},
the associated primes of $I$ are of the form $P_i + Q_j$ for all $i,j$.
Thus,
\begin{align*}
  I^{(m)} &= \bigcap\limits_{i,j} (P_i+Q_j)^m = \bigcap\limits_{i,j} (P_i^m + P_i^{m-1} Q_j + \cdots + P_i Q_j^{m-1} + Q_j^m)\\
  &= \bigcap\limits_i P_i^m + \bigcap\limits_{i,j} P_i^{m-1} Q_j + \cdots + \bigcap\limits_{i,j} P_i Q_j^{m-1} + \bigcap_j Q_j^m \\ 
  &= I_1^{(m)} + I_1^{(m-1)} I_2 + \cdots + I_1 I_2^{(m-1)} + I_2^{(m)},
\end{align*}
where the third equality follows by use of the modular law for monomial
ideals (i.e. $L\cap(J+K) = L\cap J + L\cap K$ for monomial ideals $J,K,L$)
and the last equality follows from Lemma \ref{LemmaFactorSymbolicPowerProduct}.
\end{proof}

We give to applications for this theorem.

\begin{corollary}
If $G_1$ and $G_2$ are disjoint graphs, then
\[
I(G_1\cup G_2)^{(m)} = \sum\limits_{j=0}^m I(G_1)^{(m-j)} I(G_2)^{(j)}.
\]
\end{corollary}

As a second application, we are able to give a new proof for Corollary \ref{disjoint}. Unlike our first proof, which is combinatorial, the methods of this new proof are entirely algebraic.

\begin{corollary}\label{disjoint2}
Suppose that $I$ and $J$ are two squarefree monomial ideals
of the ring $R = k[x_1,\ldots,x_n,y_1,\ldots,y_n]$.
Furthermore, suppose that $I$ is generated
by monomials only in the $x_i$'s and
$J$ is generated by monomials only in the $y_j$'s.
Then
\[\widehat\alpha(I+J) = \min\{\widehat\alpha(I),\widehat\alpha(J)\}.\]
\end{corollary}

\begin{proof}
Since $I^{(n)} \subseteq (I+J)^{(n)}$, we deduce that $\alpha(I^{(n)}) \ge \alpha ((I+J)^{(n)})$ for all $n\geq 0$. Thus we also have the inequality $\widehat\alpha(I) \ge \widehat\alpha(I+J)$ and similarly we obtain $\widehat\alpha(J) \ge \widehat\alpha(I+J)$. For the reverse inequality, from Theorem \ref{LemmaSymbolicBinomialTheorem}, we deduce that 
$$\alpha((I+J)^{(n)}) = \min_{0\leq m\leq n} \alpha(I^{(m)} J^{(n-m)}) \ge \min_{0\leq m\leq n} (m \widehat \alpha (I) + (n-m) \widehat \alpha(J)) \ge n \cdot \min(\widehat \alpha(I), \widehat\alpha (J)),$$
which allows to conclude $\widehat\alpha(I+J)\geq \min(\widehat \alpha(I), \widehat\alpha (J))$.
\end{proof}

\subsection{Monomial ideals of mixed height}\label{mixedeight}  For many families of ideals $I$
for which we know $\widehat\alpha(I)$ (or $\rho_a(I)$ and $\rho(I)$), the ideal
$I$ is unmixed, i.e., all of the associated primes of $I$ have the same height.  In this  final part of the paper, we present some initial 
results on the problem of computing
these invariants for ideals of mixed height.  

In particular, as a case study
we focus on the 
scheme $Z \subset \mathbb P^n$ $(n\geq 2)$ defined by $n$ general $(n-2)$--planes in an 
$(n-1)$--plane and one point out of the plane (a ``monomial star'').   This scheme
is similar to the type of varieties studied by Fatabbi, Harbourne, and Lorenzini
\cite{FHL}.  Specifically, we wish to consider the family of ideals
\begin{align*}
I_{Z} & = (x_0x_n, x_1x_n, \ldots, x_{n-1}x_n, x_0x_1 \cdots x_{n-1}) \\
            & = (x_0 \cdots x_{n-1},x_n) \cap (x_0, \ldots, x_{n-1}) = P_0 \cap P_1 \cap \cdots \cap P_{n-1} \cap P_n,
\end{align*}
where for $0 \leq i \leq n-1$
$P_i = (x_i, x_n)$ for $i=1,\ldots,n$ and $P_n = (x_0, \ldots, x_{n-1}).$

We wish to obtain information about containments of the form
$I_{Z}^{(m)} \subseteq I_{Z}^r$.  Theorem \ref{mainresult} enables us
to easily compute the Waldschmidt constant for these ideals, and consequently,
a lower bound on the (asymptotic) resurgence.

\begin{lemma}
With $Z \subseteq \mathbb{P}^n$ defined as above, we have
$$\widehat{\alpha}(I_{Z}) = \frac{2n-1}{n}.$$
\end{lemma}

\begin{proof}
The second statement follows from Theorem \ref{resurgencebounds}. 
We know that 
$\displaystyle \widehat{\alpha}(I_{Z})$ is the value of the LP obtained by minimizing ${\bf b}^T{\bf y} = {\bf 1}^T {\bf y}$ subject to $A{\bf y} \geq {\bf 1} = {\bf c}$ and ${\bf y} \geq {\bf 0}$ where
$$A = \left[ \begin{array}{ccccccc}
1 & 0 & 0 & \cdots & 0 & 0 & 1\\
0 & 1 & 0 & \cdots & 0 & 0 & 1\\
\vdots & \vdots & \vdots & \vdots & \vdots & \vdots & \vdots \\
0 & 0 & 0  & \cdots & 0 & 1 & 1 \\
1 & 1 & 1 & \cdots & 1 & 1 & 0
\end{array}
\right].$$
A feasible solution is
\[{\bf y}^T = \begin{bmatrix} \frac{1}{n} & \frac{1}{n} & \cdots & \frac{1}{n}
& \frac{n-1}{n} \end{bmatrix}.\]
To see that this solution is optimal, note that the matrix $A$ is symmetric and ${\bf b} = {\bf c} = {\bf 1}$, and so this solution is also a feasible solution to the dual linear program.
Therefore,
$\displaystyle \widehat{\alpha}(I_{Z}) = {\bf 1}^T {\bf y} = (2n-1)/n.$
\end{proof}

We will compute the resurgence for the ideals $I_Z$. 



\begin{lemma}\label{powersymb}
Let $0 \neq I \subset R$ be a squarefree monomial ideal
such that $I^{(m)} = I^m$ for all $m \geq 0$.  Let
$y$ be a new variable and 
consider the ideal $(I, y)$ in $S = R[y]$. Then 
$$(I,y)^{(m)} = (I, y)^m ~~\mbox{for all $\geq 0$}.$$
\end{lemma}

\begin{proof}
By Theorem \ref{LemmaSymbolicBinomialTheorem}, we have
$$(I,y)^{(m)} = \sum_{m_1+m_2=m}I^{(m_1)}(y)^{(m_2)} =
\sum_{m_1+m_2=m}I^{m_1}y^{m_2} = (I,y)^m.$$
\end{proof}

We now collect together a number of results that we will require regarding the 
symbolic and ordinary powers of $I_Z$.

\begin{lemma}\label{gensIZ} 
With $Z \subseteq \mathbb{P}^n$ defined as above, we have
\begin{enumerate}
\item[$(i)$] $I_{Z}^{(m)} = (x_0x_1 \cdots x_{n-1},x_n)^m
\cap (x_0,x_1, \ldots,x_{n-1})^m$.
\item[$(ii)$] 
\footnotesize
$I_Z^{(m)} =  \sum_{i=0}^{\lfloor \frac{(n-1)m}{n}\rfloor}
x_{n}^i(x_{0}\cdots x_{n-1})^{m-i} + \sum_{i=\lfloor
\frac{(n-1)m}{n}\rfloor + 1}^m x_{n}^i(x_{0}\cdots x_{n-1})^{m-i}
(x_{0},\ldots, x_{n-1})^{ni-(n-1)m}$
\normalsize
\item[$(iii)$] $I^s = \sum_{i=0}^s x_n^i (x_{0}\cdots x_{n-1})^{s-i} (x_{0},\ldots, x_{n-1})^i$.
\end{enumerate}
\end{lemma}
\begin{proof} 
For $(i)$, by Lemma \ref{powersymb} we have 
\begin{eqnarray*} I_{Z}^{(m)} &=& (x_0,x_n)^m \cap
(x_1,x_n)^m \cap \ldots\cap (x_{n-1},x_n)^m \cap (x_0, \ldots,
x_{n-1})^m \\
&= &(x_0 \cdots x_{n-1},x_n)^{(m)} \cap (x_0, \ldots, x_{n-1})^m =
(x_0 \cdots x_{n-1},x_n)^m \cap (x_0, \ldots, x_{n-1})^m.
\end{eqnarray*}
For $(ii)$ we have
\footnotesize
\begin{eqnarray*}
I^{(m)} & =& (x_{0}\cdots x_{n-1},x_n)^m \cap (x_{0},\ldots, x_{n-1})^m = \sum_{i=0}^{m} \left ( x_n^i(x_{0}\cdots x_{n-1})^{m-i}
\cap (x_{0},\ldots, x_{n-1})^m \right )\\
& =& \sum_{i=0}^{\lfloor (n-1)m/n\rfloor} x_n^i(x_{0}\cdots
x_{n-1})^{m-i} +
\sum_{i=\lfloor (n-1)m/n\rfloor+1}^{m} {x_n^i (x_{0}\cdots x_{n-1})^{m-i}(x_{0},\ldots, x_{n-1})^{m-n(m-i)}}\\
& = &\sum_{i=0}^{\lfloor (n-1)m/n\rfloor} x_n^i(x_{0}\cdots
x_{n-1})^{m-i} +
\sum_{i=\lfloor (n-1)m/n\rfloor+1}^{m} {x_n^i (x_{0}\cdots x_{n-1})^{m-i}(x_{0},\ldots, x_{n-1})^{ni-(n-1)m}}.\\
\end{eqnarray*}
\normalsize
Finally, for $(iii)$, we have
\begin{eqnarray*}
I^s & =& (x_{0}\cdots
x_{n-1},x_0x_n,x_1x_n,\ldots,x_{n-1}x_{n})^s\\
& =& \sum_{i=0}^s (x_{0}\cdots x_{n-1})^{s-i} (x_n(x_{0},\ldots,
x_{n-1}))^i =\sum_{i=0}^s x_n^i (x_{0}\cdots x_{n-1})^{s-i}
(x_{0},\ldots, x_{n-1})^i.
\end{eqnarray*}
\end{proof}

\begin{theorem}
With $Z \subseteq \mathbb{P}^n$ defined as above, let $I = I_Z$.  Then
\begin{enumerate}
\item[$(i)$] if $m,s$ are positive integers with $\frac{m}{s}\geq \frac{n^2}{n^2-n+1}$, then $I^{(m)} \subseteq I^{s}$
\item[$(ii)$] $I^{(n^2k)} \not\subseteq I^{(n^2-n+1)k+1}$, for all integers $k\geq 0$.
\end{enumerate}
Consequently, $\rho(I_{Z}) = \rho_a(I_{Z}) = \frac{n^2}{n^2-n+1}$.
\end{theorem}

\begin{proof} For $(i)$ it suffices to prove that each 
term in the decomposition of $I^{(m)}$ of Lemma \ref{gensIZ} is contained in 
$I^{s}$. By Lemma \ref{gensIZ} we have to analyze the following cases:

\noindent
{\it Case 1}: if $0\le i \le \lfloor \frac{(n-1)m}{n}\rfloor$, we will prove that 
$$x_n^i(x_0 \cdots x_{n-1})^{m-i} \in x_n^i 
(x_0\cdots x_{n-1})^{s-i}(x_0, \ldots,x_{n-1})^i.$$
This is equivalent to $(x_0\cdots x_{n-1})^{m-s} \in (x_0, \ldots,x_{n-1})^i$, which is further equivalent to $n(m-s) \ge i$. To prove the latter inequality holds, it is sufficient to show that $n(m-s)\ge \frac{(n-1)m}{n}$. Elementary manipulations show that this inequality is equivalent to the hypothesis $\frac{m}{s}\geq \frac{n^2}{n^2-n+1}$.

\noindent
{\it Case 2:} if $\lfloor \frac{(n-1)m}{n}\rfloor +1 \le i \le \min\{m,s\}$, we will prove that
$$x_n^i(x_0 \cdots x_{n-1})^{m-i}(x_0,\ldots,x_{n-1})^{ni - (n-1)m} \subseteq x_n^i 
(x_0\cdots x_{n-1})^{s-i}(x_0,\ldots,x_{n-1})^i.$$
This is equivalent to $(x_0\cdots x_{n-1})^{m-s} \in (x_0, \ldots,x_{n-1})^{(n-1)(m-i)}.$ 
The latter is equivalent to $n(m-s) \ge (n-1)(m-i)$, or $m+(n-1)i\geq ns$, which holds true
since $m+(n-1)i > m+(n-1) \cdot \frac{(n-1)m}{n}=m\cdot \frac{n^2-n+1}{n}\geq ns$. The last inequality uses  the hypothesis $\frac{m}{s}\geq \frac{n^2}{n^2-n+1}$.

\noindent
{\it Case 3:} if $ \min\{m,s\} < i \le m$, we will prove that
$$x_n^i(x_0\cdots x_{n-1})^{m-i}(x_0,\ldots,x_{n-1})^{ni - (n-1)m} \subseteq x_n^{s} 
(x_0,\ldots,x_{n-1})^{s}.$$
For that, it suffices to prove that $n(m-i) \ge s +(n-1)m -ni$, which is equivalent to 
$m \ge s$. Either this inequality is satisfied or else this case is vacuous. 

For $(ii)$ the following notation will be used in the proof. For a monomial 
$f$, $\deg_1f$ denotes the total degree of the monomial part involving the variables 
$x_0,\ldots, x_{n-1}$, while $\deg_2f$ denotes 
exponent of $x_n$ in $f$. 
For two monomials $f$ and $g$, if $f$ is a multiple of $g$, then clearly
$\deg_1f\ge \deg_1 g, \text{ and } \deg_2f\ge \deg_2g.$

Consider the monomial $f = (x_0\cdots x_{n-1})^{nk} x_n^{(n-1)nk}$ in $I^{(n^2k)}$. 
Now assume that $f\in I^{s}$, where $s = (n^2-n+1)k+1$. Then there exists a 
minimal generator $g$ in $I^{s}$ such that $g|f$. Such a monomial $g$ has the 
form $(x_0\cdots x_{n-1})^{s-\alpha} x_0^{t_0}\cdots x_{n-1}^{t_{n-1}}x_n^\alpha$, where 
$t_0,\ldots, t_{n-1}$ are non-negative integers such that $t_0+ \cdots + t_{n-1} = \alpha$.
 We have
$\deg_2 f = (n-1)nk \ge \deg_2 g = \alpha.$ 
Furthermore,
$\deg_1 f = n^2 k \ge \deg_1g = n(s-\alpha) +\alpha.$ 
In particular, this implies
$$\alpha\ge \frac{n(s-nk)}{n-1} > \frac{n(n^2-2n+1)k}{n-1} = n(n-1)k,$$
which is a contradiction, so $(ii)$ follows.

The non-containment relations in part $(ii)$ yield $$\rho(I) \geq \sup \left\{ \frac{n^2k}{(n^2-n+1)k + 1} ~\vert~ k>0\right\}=\lim_{k \to \infty} \frac{n^2k}{(n^2-n+1)k + 1} = \frac{n^2}{n^2-n+1},$$ while part $(i)$ shows that the opposite inequality holds. We conclude that  $\rho(I_{Z}) = \frac{n^2}{n^2-n+1}$. 

As for the asymptotic resurgence, since $\rho_a(I)\leq \rho(I)$, we have $\rho_a(I)\leq \frac{n^2}{n^2-n+1}$. From part $(ii)$, since $I^{((n^2-n+1)k+1)t}\subseteq I^{((n^2-n+1)kt+1}$ for $t\geq 1$,  we deduce
$$I^{(n^2kt)} \not\subseteq I^{((n^2-n+1)k+1)t} \text{ for } t\geq 1.$$
It follows that 
$$\rho_a(I) \geq\lim_{k \to \infty} \frac{n^2k}{(n^2-n+1)k + 1} = \frac{n^2}{n^2-n+1},$$
allowing us to conclude that  $\rho_a(I_{Z}) = \frac{n^2}{n^2-n+1}$. 
\end{proof}

\begin{remark}
Similar to Remark \ref{nonequal1}, we have the following inequalities
\begin{eqnarray*}
\frac{\alpha(I_Z)}{\widehat\alpha(I_Z)} = \rho_a(I_Z)&=&\rho(I_Z)= \frac{\omega(I_Z)}{\widehat\alpha(I_Z)}=\frac{4}{3}, \text{ if } n=2\\
\frac{2n}{2n-1} =\frac{\alpha(I_Z)}{\widehat\alpha(I_Z)} < \rho_a(I_Z)&=&\rho(I_Z)< \frac{\omega(I_Z)}{\widehat\alpha(I_Z)}=\frac{n^2}{2n-1}, \text{ if } n>2
\end{eqnarray*}
for the family of ideals $I_Z$. The case $n=2$ corresponds to $Z$ being a set of 3 points in $\mathbb{P}^2$.  
\end{remark}


\end{document}